\documentclass[11pt]{amsart}
\usepackage[hmarginratio=1:1]{geometry}                
\usepackage{graphicx}
\usepackage{mathbbol}
\usepackage{amsmath}%
\usepackage{amsfonts}%
\usepackage{amssymb}
\usepackage{amsthm}
\usepackage[mathscr]{euscript}
\usepackage[utf8]{inputenc}
\usepackage{cite}
\usepackage{hyperref}
\usepackage[capitalize]{cleveref}
\usepackage[svgnames]{xcolor}
\usepackage{tikz}
\usetikzlibrary{matrix,arrows,decorations.markings}
\pgfdeclarelayer{edgelayer}
\pgfdeclarelayer{nodelayer}
\pgfsetlayers{edgelayer,nodelayer,main}

\tikzstyle{none}=[inner sep=0pt]

\tikzstyle{hidden}=[rectangle,fill=White,draw=White]
\tikzstyle{antipode}=[circle,fill=White,draw=Black]
\tikzstyle{delta}=[up triangle,fill=White,draw=Black]
\tikzstyle{nabla}=[down triangle,fill=White,draw=Black]
\tikzstyle{up4}=[up4vertex,fill=White,draw=Black]
\tikzstyle{down4}=[down4vertex,fill=White,draw=Black]
\tikzstyle{square}=[rectangle,fill=White,draw=Black]
\tikzstyle{curPt}=[circle,fill=White,draw=Black]
\tikzstyle{arrow}=[->]
\tikzstyle{hookedarrow}=[right hook->]
\tikzstyle{dashedarrow}=[dashed,->]

\theoremstyle{plain}
\newtheorem{Theorem}{Theorem}[section]
\newtheorem{Lemma}{Lemma}[section]
\newtheorem{Corollary}{Corollary}[section]
\newtheorem{Proposition}{Proposition}[section]

\theoremstyle{definition}
\newtheorem{Definition}{Definition}[section]

\newtheorem{Example}{Example}[section]

\theoremstyle{remark}
\newtheorem{remark}{Remark}[section]

\numberwithin{equation}{section}

\crefname{pluralequation}{Eqs.}{Eqs.}
\Crefname{pluralequation}{Eqs.}{Eqs.}

\newcommand{\mbb}{\mathbb}

\newcommand{\mc}{\mathcal}
\newcommand{\on}{\operatorname}

\newcommand{\Cat}{{\mathscr{C}\mathrm{at}}}
\newcommand{\CSS}{CSS}

\newcommand{\Pre}{\mathscr{P}}
\newcommand{\Spc}{\mathscr{S}}
\newcommand{\LTop}{{\mathscr{LT}\mathrm{op}}}

\newcommand{\colim}{{\on{co}\lim}}

\newcommand{\BDelta}{\mathbb{\Delta}}
\newcommand{\BSigma}{\mathbb{\Sigma}}
\newcommand{\BLambda}{\mathbb{\Lambda}}
\newcommand{\ccorner}{\mathbin{\rotatebox[origin=c]{-135}{$\llcorner$}}}
\providecommand{\abs}[1]{\lvert#1\rvert}

\newcommand{\mmat}[2][3em]{\matrix (#2) [matrix of math nodes, row sep=#1,
  column sep=#1, text height=1.5ex, text depth=0.25ex]}
\tikzset{node distance=2cm, auto}

\DeclareFontFamily{U}{min}{}
\DeclareFontShape{U}{min}{m}{n}{<-> udmj30}{}
\newcommand\yon{\!\text{\usefont{U}{min}{m}{n}\symbol{'210}}\!}
\newcommand\yoneda{\text{\usefont{U}{min}{m}{n}\symbol{'210}\symbol{'155}\symbol{'140}}}

\title{The stack of higher internal categories and stacks of iterated spans.}
\author{David Li-Bland}

\begin{document}
\begin{abstract}

In this paper, we show that two constructions form stacks:
Firstly, as one varies the $\infty$-topos, $\mc{X}$, Lurie's homotopy theory of higher categories internal to $\mc{X}$ varies in such a way as to form a stack over the $\infty$-category of all $\infty$-topoi. 

Secondly, 
we show that Haugseng's construction of the higher category of iterated spans  in a given $\infty$-topos (equipped with local systems) can be used to define various stacks over that $\infty$-topos.

As a prerequisite to these results, we discuss properties which limits of $\infty$-categories inherit from the $\infty$-categories comprising the diagram. For example,  Riehl and Verity have shown that possessing (co)limits of a given shape is hereditary. Extending their result somewhat, we show that possessing Kan extensions of a given type is heriditary, and more generally that the adjointability of a functor is heriditary.

\end{abstract}
\maketitle
\tableofcontents
\section{Introduction}

Building upon the ideas of Rezk \cite{Rezk:2001ey} and Barwick \cite{Barwick:2005wl,Barwick:2011wl}, Lurie constructed a model for the homotopy theory of higher categories internal to an $\infty$-topos $\mc{X}$ \cite{Lurie:2009uz}. More precisely, he constructs an $\infty$-category $\CSS_k(\mc{X})$ of complete $k$-fold Segal objects in an arbitrary $\infty$-topos $\mc{X}$. Our first result (cf. Theorem~\ref{thm:CSSaSheaf}) in this paper is to show that the construction \begin{equation}\label{eq:CSSfunctor}\mc{X}\mapsto \CSS_k(\mc{X})\end{equation} satisfies a certain descent condition: 
suppose that $\mc{X}_i\to \mc{X}$ is an \'etale cover of $\mc{X}$ indexed by a small simplicial set $i\in I$, i.e. $$\mc{X}\cong \colim_{i\in I}\mc{X}_i,$$
then $$\CSS_k(\mc{X})\cong \lim_{i\in I}\CSS_k(\mc{X}_i).$$
In other words, \eqref{eq:CSSfunctor} defines a stack (cf. \cite[Notation~6.3.5.19]{Lurie:2009un}).

Given an $\infty$-category $\mc{C}$, Barwick \cite{Barwick:2013th} showed how to construct an $(\infty,1)$-category $\on{Span}(\mc{C})$ which has the same space of objects as $\mc{C}$, but whose morphisms between two objects $c_0,c_1\in \mc{C}$ is the space of diagrams in $\mc{C}$ of the form
$$\begin{tikzpicture}
 \node  (m-1-1)  at (0,1) {$x$};
 \node  (m-1-2)  at (1,0) {$c_0$};
 \node  (m-2-1)  at (-1,0) {$c_1$};
\draw[->] (m-1-1) -- (m-1-2);
\draw[->] (m-1-1) -- (m-2-1);
\end{tikzpicture}$$
That is, spans $c_0\nrightarrow c_1$ in $\mc{C}$.
Composition of two such morphisms is given by taking the fibred product. Given an $\infty$-topos $\mc{X}$ and a complete $k$-fold Segal object $X_{\bullet,\dots,\bullet}\in\CSS_k(\mc{X})$, Haugseng extends this construction in \cite{Haugseng:2014vw} to produce an $(\infty,k)$-category $\on{Span}_k(\mc{X},X_{\bullet,\dots,\bullet})\in\CSS_k(\hat\Spc)$ of iterated $k$-fold spans with local systems valued in $X_{\bullet,\dots,\bullet}\in\CSS_k(\mc{X})$ (here $\hat\Spc$ is the $\infty$-category of (not necessary small) spaces). 

Our second main result is to show that the assignment of $\on{Span}_k(\mc{X},X_{\bullet,\dots,\bullet})$ to $X_{\bullet,\dots,\bullet}\in\CSS_k(\mc{X})$ depends continuously on $\mc{X}$ and $X_{\bullet,\dots,\bullet}$ (i.e. it preserves small limits). This result is somewhat more subtle than it first appears: for example, the functor
$$\CSS_k(\mc{X})\xrightarrow{X_{\bullet,\dots,\bullet}\mapsto \on{Span}_k(\mc{X},X_{\bullet,\dots,\bullet})}\CSS_k(\hat\Spc)$$
is not continuous - it preserves neither products nor the terminal object. To correctly understand the continuity of $\on{Span}_k$, we need to work in a larger context: we assemble all the $\infty$-categories $\CSS_k(\mc{X})$ into one large $\infty$-category $\int\CSS_k$, whose (roughly speaking)\footnote{More precisely, $\int\CSS_k$ is the lax colimit of the functor $\mc{X}\mapsto \CSS_k(\mc{X})$, (cf. \cite{Gepner:2015ww}); equivalently, $\int\CSS_k$ is Lurie's unstraightening of that functor. We provide a direct construction of $\int\CSS_k$, however.}
\begin{itemize}
	\item objects are pairs $(\mc{X},X_{\bullet,\dots,\bullet})$, where $\mc{X}$ is an arbitrary $\infty$-topos and $X_{\bullet,\dots,\bullet}\in \CSS_k(\mc{X})$ is a higher category internal to $\mc{X}$, and
	\item morphisms $(\mc{X},X_{\bullet,\dots,\bullet})\to (\mc{Y},Y_{\bullet,\dots,\bullet})$ consist of a geometric morphism of $\infty$-topoi, $(f^*\dashv f_*):\mc{X}\leftrightarrows\mc{Y}$, together with a morphism $X_{\bullet,\dots,\bullet}\to (f_*)_!Y_{\bullet,\dots,\bullet}$ in $\CSS_k(\mc{X})$.
\end{itemize}
In Theorem~\ref{thm:SpanCont}, we prove that the functor
$$\int\CSS_k\xrightarrow{(\mc{X},X_{\bullet,\dots,\bullet})\mapsto \on{Span}_k(\mc{X},X_{\bullet,\dots,\bullet})}\CSS_k(\hat\Spc)$$
is continuous (preserves small limits).

Consequently,
suppose that for every $U\in \mc{X}$ we assign (in a natural way) a complete Segal object $\sigma(U)\in \CSS_k(\mc{X}_{/U})$, in a manner which depends locally on $U\in\mc{X}$: that is, for any colimit diagram $U_i\to U$ in $\mc{X}$ indexed by a small simplicial set $i\in I$, $$\sigma(U)=\lim_{i\in I}\sigma(U_i)$$
(where the latter limit is taken in $\int\CSS_k$), then 
\begin{subequations}\label{eq:IntroSpanStack}
\begin{equation}\mc{X}^{op}\xrightarrow{U\mapsto \on{Span}_k(\mc{X}_{/U},\sigma(U))}\CSS_k(\hat\Spc)\end{equation}
defines an $(\infty,k)$-stack over $\mc{X}$ (cf. Theorem~\ref{thm:StackofSpans}).

As a first example, $\on{Span}_k(\mc{X},X_{\bullet,\dots,\bullet})$ itself forms a stack \begin{equation}\mc{X}^{op}\xrightarrow{U\mapsto \on{Span}_k(\mc{X}_{/U},U\times X_{\bullet,\dots,\bullet})}\CSS_k(\hat\Spc)\end{equation}
\end{subequations}
over $\mc{X}$. As a second example, taking $\mc{X}=\on{dSt}_{\mbb{K}}$ to be derived stacks over a field $\mbb{K}$ of characteristic zero, and $X_{\bullet,\dots,\bullet}$ to be trivial, the fact that \eqref{eq:IntroSpanStack} forms a stack implies that the \emph{derived} composition of spans depends continuously (algebraically, in fact) on the spans involved.

Our motivation for these results comes from mathematical physics: the success of the \emph{Lagrangian Creed}:
\begin{center} \it ``everything is a Lagrangian correspondence''\footnote{Paraphrased from \cite{Weinstein:1979vn} ``everything is a Lagrangian submanifold''.},\end{center}
places Lagrangian correspondences between symplectic manifolds at the centre of classical mechanics. Lagrangian correspondences have two major flaws however: \textbf{firstly} they fail to compose in general, i.e. given two Lagrangian correspondences $$U\overset{L}{\nleftarrow} V\text{ and }V\overset{L'}{\nleftarrow} W,$$
their set theoretic composite \begin{equation}\label{eq:LagComp}U\overset{L\circ L'}{\nleftarrow} W,\end{equation} often fails to be smooth, and - \textbf{secondly} - when the composite \eqref{eq:LagComp}  exists as a Lagrangian correspondence, it may not depend continuously on $L$ and $L'$.

The first of these issues was essentially resolved by Pantev, To\"en, Vaqui\'e, and Vessozi \cite{Pantev:2011uz}, and Calaque \cite{Calaque:2013un} using derived geometry. Building upon this, Haugseng \cite{Haugseng:2014vw} then gave an embedding of Weinstein's symplectic `category' \cite{Weinstein82} whose morphisms are the Lagrangian correspondences, as a subcategory of $\on{Span}_1(\on{dSt}_{\mbb{K}}, \mc{A}^2_{\on{cl}})$, spans of derived stacks with local systems valued in closed 2-forms. The fact that \eqref{eq:IntroSpanStack} is a stack is a first step towards a deeper understanding of what it means to restore the continuity of composition using derived geometry.

Moreover, the second issue - the failure of composition to be continuous - is closely related to the failure to quantize classical mechanics functorially: After quantizing pairs where the composite \eqref{eq:LagComp} fails to depend continuously on $L$ and $L'$, one is typically trying to multiply Dirac $\delta$-functions in the corresponding quantization. In work in progress with Gwilliam, Haugseng, Johnson-Freyd, Scheimbauer, and Weinstein, we show that \emph{at least to first order} (i.e. after linearizing),
\begin{itemize}
	\item the \emph{derived} composition of Lagrangian correspondences \eqref{eq:LagComp} depends continuously on $L$ and $L'$ (cf. \cite{LiBland:2014tf}), and
	\item there is a functorial quantization.
\end{itemize}

In order to show that \eqref{eq:CSSfunctor} and \eqref{eq:IntroSpanStack} define stacks, we first need to examine limits of $\infty$-categories. Suppose that $\mc{C}_k,\quad k\in K$ is some diagram of $\infty$-categories indexed by a simplicial set $K$. In \cite{Riehl:2014ut}, Riehl and Verity show that if each $\mc{C}_k$ has all (co)limits of shape $I$ (where $I$ is some small simplicial set), and for each arrow $k\to k'$ in $K$, the corresponding functor $\mc{C}_k\to \mc{C}_{k'}$ preserves all (co)limits of shape $I$, then the limit $\infty$-category $\lim_{k\in K}\mc{C}_k$ also has all (co)limits of shape $I$. After providing an alternate proof of this result (cf. Theorem~\ref{thm:LimInLim}), we extend their result to show that if each $\mc{C}_k$ possesses all Kan extensions along a functor $\mc{I}\to\mc{I}'$, and each functor $\mc{C}_k\to \mc{C}_{k'}$ preserves those Kan extensions, then the limit $\infty$-category $\lim_{k\in K}\mc{C}_k$ also has all Kan extensions along $\mc{I}\to\mc{I}'$ (cf. Corollary~\ref{cor:KanClosed}). More generally, suppose that 
	$$F_k:\mc{C}_k \rightleftarrows\mc{D}_k:G_k, \quad k\in K$$ is diagram of adjunctions 
	coherently indexed by a small simplical set $K$, then we prove there is an adjunction $$\lim_{k\in K}F_k:\lim_{k\in K}\mc{C}_k \rightleftarrows\lim_{k\in K}\mc{D}_k:\lim_{k\in K}G_k$$ between the corresponding limit $\infty$-categories (cf. Theorem~\ref{thm:limadj}).
	
	\subsection{Acknowledgements}
	We would like to thank Owen Gwilliam, Theo Johnson-Freyd, Claudia Scheimbauer, and Alan Weinstein for many important discussions surrounding the content of this paper. We would also like to thank Thomas Nikolaus for a very helpful introduction to descent theory, and Omar Antol{\'\i}n Camarena for a number of helpful conversations about higher category theory. The author was supported by an NSF Postdoctoral Fellowship DMS-1204779.
	
	\subsection{Notation}
	We generally use the notation and terminology developed by Lurie (cf. \cite{Lurie:2009un}). In particular, by an \emph{$\infty$-category}, we mean a quasicategory, i.e. a simplicial set satisfying certain horn filling conditions. In addition, we use the following notation, some of which differs from Lurie's:
	\begin{itemize}
		\item $\BDelta$ denotes the simplicial indexing category whose objects are non empty fnite totally ordered sets $[n] := \{0, 1, . . . , n\}$ and morphisms are order-preserving functions between them. $\Delta^n:\BDelta^{op}\to \on{Sets}$ is the simplicial set represented by $[n]$.
		\item We denote generic $\infty$-categories by upper-case caligraphic letters, $\mc{A},\mc{B},\mc{C},\mc{D},\dots$. We typically denote elements $c\in \mc{C}$ of a generic $\infty$-category by lowercase versions of the same letter.
		\item We let $\on{Fun}(\mc{C},\mc{D})$ denote the $\infty$-category of functors between $\infty$-categories, and $\on{Map}_{\mc{C}}(c,c')$ denote the mapping space between two objects $c,c'\in \mc{C}$.
		\item If $\mc{C}$ is an $\infty$-category, we write $\iota\mc{C}$ for the \emph{interior} or \emph{classifying space of objects} of $\mc{C}$, i.e. the maximal Kan complex contained in $\mc{C}$.
		\item If $f : \mc{C} \to \mc{D}$ is left adjoint to a functor $g : \mc{D} \to\mc{C}$, we will refer to the adjunction as $f \dashv g$.
		\item $\Cat_\infty$ denotes the $\infty$-category of small $\infty$-categories, and the $\infty$-category of spaces, $\Spc\subset \Cat_\infty$, is the full subcategory spanned by the Kan complexes.
		\item If $\mc{C}$ is an $\infty$-category, we let $$\yon:\mc{C}\to \Pre(\mc{C}):=\on{Fun}(\mc{C},\Spc)$$
		denote the Yoneda embedding\footnote{following Theodore Johnson-Freyd's suggestion, we denote the Yoneda embedding by the first Hiragana character - pronounced `yo' - of his name, \yoneda.}.
		\item Suppose that $p_0:X_0\to K$ and $p_1:X_1\to K$ are two morphisms of simplicial sets
		and that $p_1$ is a (co)Cartesian fibration.
			 We let $\on{Fun}_{K}(X_0,X_1)$ denote the simplicial subset of all simplicial maps between $X_0$ and $X_1$ spanned by those maps which intertwine $p_0$ and $p_1$. Note that $\on{Fun}_{K}(X_0,X_1)$ is automatically an $\infty$-category (cf. \cite[Remark~3.1.3.1]{Lurie:2009un}).
			 
			  When $p_0$ and $p_1$ are both (co)Cartesian fibrations,
		then  we let $$\on{Fun}^{\on{(co)Cart}}_K(X_0,X_1)\subseteq \on{Fun}_{K}(X_0,X_1)$$ denote the subcategory spanned by those maps which preserve the (co)Cartesian edges.\footnote{Note that $\on{Fun}_{K}(X_0,X_1)$ is denoted by $\on{Map}_{K}^\flat(X_0^\flat,X_1^\natural)$ in \cite{Lurie:2009un}, while $\on{Fun}^{\on{(co)Cart}}_K(X_0,X_1)$ is denoted by $\on{Map}_{K}^\flat(X^\natural_0,X^\natural_1)$. We choose this alternate notation (in line with \cite{Gepner:2015ww}) to emphasize that the resulting simplicial set is an $\infty$-category.}
	\end{itemize}
 
\section{Properties Inherited by Limit $\infty$-Categories}
Let $\Cat_\infty$ denote the $\infty$-category of small $\infty$-categories, let $K$ be a small simplicial set, and consider a diagram $p':K\to \Cat_\infty$. We will be interested in the limit $\infty$-category, $\lim p'\in \Cat_\infty$. 

To compute such limits, consider the functor
\begin{subequations}
\begin{equation}\label{eq:DiagFunct}\Delta:\Cat_\infty\to \on{Fun}(K,\Cat_\infty),\end{equation} which sends an $\infty$-category $\mc{C}\in \Cat_\infty$ to the constant diagram: 
$$\Delta_{\mc{C}}:k\to \mc{C},\quad\text{ for any }k\in K.$$
The right adjoint to \eqref{eq:DiagFunct} is the functor which sends a diagram $p'\in \on{Fun}(K,\Cat_\infty)$ to the corresponding limit $\infty$-category $\lim p'\in \Cat_\infty$.

Now, fix a second diagram $p_0'\in \on{Fun}(K,\Cat_\infty)$ and consider the functor
\begin{equation}\label{eq:DiagTensorFunctor}\Delta\times p_0':\Cat_\infty \to \on{Fun}(K,\Cat_\infty),\end{equation}\end{subequations}
which sends any $\infty$-category $\mc{C}\in \Cat_\infty$ to the functor
$$\Delta_\mc{C}\times p_0':k\to \mc{C}\times p_0'(k),\quad\text{ for any }k\in K.$$
The right adjoint to $\Delta\times p_0'$ sends any diagram $p'\in  \on{Fun}(K,\Cat_\infty)$ to the $\infty$-category $\on{Nat}_{K}(p_0',p')\in \Cat_\infty$ of natural transformations between $p_0'$ and $p'$. Since in the special case that $p_0'=\Delta_\ast$ is the constant diagram at the terminal $\infty$-category\footnote{The terminal $\infty$-category, $\ast$, has exactly one object, one 1-morphism (the identity) and one $n$-morphism for every $n$.} both functors \eqref{eq:DiagTensorFunctor} and \eqref{eq:DiagFunct} coincide, in particular, for any diagram $p'\in \on{Fun}(K,\Cat_\infty)$, we have an equivalence
$$\lim p'\cong \on{Nat}_{K}(\Delta_\ast,p')$$ between the limit of $p'$ and the $\infty$-category of natural transformations from the trivial diagram to $p'$.

In practice, often the best description of diagrams in $\Cat_\infty$ is in terms of (co)Cartesian fibrations, as developed by Lurie \cite[\S~2.4]{Lurie:2009un}. Briefly, given an inner fibration between $\infty$-categories $\mc{C}\to \mc{D}$ a edge $f:c\to c'$ in $\mc{C}$ is called \emph{coCartesian} if 
$$\mc{C}_{f/}\to \mc{C}_{c/}\times_{\mc{D}_{p(c)/}}\mc{D}_{p(f)/}$$ is an equivalence.\footnote{Recall that for any diagram $q:I\to \mc{C}$, the \emph{undercategory}  
 $\mc{C}^{q/}$ satisfies the universal property that the space of maps $Y\to \mc{C}^{q/}$ classifies maps of the form $I\diamond Y\to \mc{C}$ (which restrict to $q$ along $I$), where $$I\diamond Y= I\coprod_{I\times Y\times\{0\}} (I\times Y\times \Delta^1)\coprod_{I\times Y\times \{1\}} Y.$$}
 The fibration $\mc{C}\to \mc{D}$ is called \emph{coCartesian} if there is a coCartesian edge over any edge in $\mc{D}$ starting at any vertex of $\mc{C}$. The fibration $\mc{C}\to\mc{D}$ is called Cartesian if $\mc{C}^{op}\to \mc{D}^{op}$ is coCartesian.

Functors $p':K\to \Cat_\infty$ correspond to coCartesian fibrations $p:X\to K$ via the straightening/unstraightening construction \cite{Lurie:2009un}; for every $k\in K$, the fibre $X_k:=p^{-1}(k)$ is equivalent to $p'(k)$, and for every edge $k\to k'$ in $K$, the corresponding functor $X_k\to X_{k'}$ is equivalent to one sending any $x\in X_k$ to the target of a coCartesian edge over $k\to k'$ starting at $x$.

Given a second coCartesian fibration $p_0:X_0\to K$ corresponding to a diagram $p_0':K\to \Cat_\infty$, Gepner, Haugseng, and Nikolaus identify the $\infty$-category of natural transformations
$$\on{Nat}_{K}(p_0',p')\cong \on{Fun}_K^{\on{coCart}}(X_0,X)$$ with the $\infty$-category of coCartesian maps $X_0\to X$, i.e. those maps $X_0\to X$ over $K$ which preserve the coCartesian edges (cf. \cite[Proposition~6.9]{Gepner:2015ww}).
In particular, Lurie shows that an elegant model of $\lim p'$ is the $\infty$-category
$$\lim p'\cong \on{Nat}_{K}(\Delta_\ast,p')\cong \on{Fun}_K^{\on{coCart}}(K,X)$$
 of coCartesian sections of $p$ \cite[Corollary~3.3.3.2]{Lurie:2009un} (note that every edge in the trivial fibration $K\to  K$ is coCartesian).
 
 \begin{Example}
 	Let $\Cat_\infty^{\on{Str}}$ denote the ordinary category whose objects are (small) $\infty$-categories, and suppose that $C$ is a (small) category. Given a strict functor $p':C\to \Cat_\infty^{\on{Str}}$, the corresponding coCartesian fibration can be computed via the relative nerve construction 
 	(cf. \cite[\S~3.2.5]{Lurie:2009un}).
 	Forming the resulting $\infty$-category of coCartesian sections, one sees that  a model for $\lim p'$ is the simplicial set whose $k$ simplices consist of the following data:
 	\begin{itemize}
 	\item for every functor $(x,y):[n]\to C\times [k]$, a choice of $n$-simplex, $\tau_{(x,y)}:\Delta^n\to p'\big(x(n)\big)$, such that
 	\end{itemize}
 	\begin{enumerate}
 	\item for every $f:[m]\to [n]$, the 
 	 following diagram commutes:
 $$\begin{tikzpicture}
 	\mmat{m}{\Delta^m & &p'\big( x(f(m))\big)\\ \Delta^n&&p'\big( x(n)\big)\\};
 	\draw[->] (m-1-1) -- node {$f$} (m-2-1);
 	\draw[->] (m-1-1) -- node {$\tau_{f^*(x,y)}$} (m-1-3);
 	\draw[->] (m-2-1) -- node {$\tau_{(x,y)}$} (m-2-3);
 	\draw[->] (m-1-3) -- node {$p'\big[x\big(f(m)\to n\big)\big]$} (m-2-3);
 \end{tikzpicture}$$
 and
 \item whenever $y(i)=y(j)$, then $\tau_{(x,y)}(\Delta^{\{i,j\}})\in p'\big(x(n)\big)$ is an equivalence.
 	\end{enumerate}
 \end{Example}
 \subsection{(co)Limits in Limit $\infty$-categories}
 Let $I$ be a second (small) simplicial set, and suppose that
 \begin{enumerate}
 	\item for each vertex $k\in K$, the $\infty$-category $p'(k)$ admits  (co)limits for all diagrams indexed by $I$.
 	\item for each edge $(k\to k')\in K$, the functor $p'(k)\to p(k')$ preserves (co)limits for all diagrams indexed by $I$.
 \end{enumerate} 
 then Riehl and Verity \cite{Riehl:2014ut} have shown that the limit $\infty$-category $\lim p' \in \Cat_\infty$ admits (co)limits for all diagrams indexed by $I$, and that those (co)limits are preserved by the functors in the limit cone. We now provide an alternate proof of this result, based on   Lurie's (co)Cartesian fibrations.
 
 To disambiguate our presentation, we will prove our results only for \emph{colimits} (rather than \emph{limits}) in the limit $\infty$-category $\lim p'$; the duality between colimits in $\lim p'$ and limits in $(\lim p')^{op}$ imply that the corresponding results hold equally for limits as well. 
 
 We begin with the special case where $I$ is the empty set, in which case we have the following variant of \cite[Proposition~2.4.4.9]{Lurie:2009un}:
\begin{Proposition}\label{lem:LimOfInit}
 Given a functor $p':K\to \Cat_\infty$, if
\begin{enumerate}
\item 	for each $k\in K$, the $\infty$-category $p'(k)$ admits an initial object $t\in X_k$, and
\item for each edge $k\to k'$ in $K$ the functor $p'(k)\to p'(k')$ preserves initial objects,
\end{enumerate}
  then 
\begin{description}
	\item[A] the limit $\infty$-category $\lim p'$ admits an initial object $t_\infty$, and
	\item[B] an object $t\in \lim p'$ is initial if and only if  for each $k\in K$, the object $\pi_k(t)\in X_k$ is initial where $\pi_k:\lim p'\to p'(k)$ is the functor appearing in the limit cone.
\end{description}
\end{Proposition}
\begin{proof}
We will find it easier to model our functor $p':K\to \Cat_\infty$ in terms of a \emph{Cartesian fibration} $p:X\to K^{op}$ (rather than a \emph{coCartesian fibration}).

Let $p:X\to K^{op}$ be a Cartesian fibration of simplicial sets classified by the functor $p':K\to \Cat_\infty$. By assumption 
\begin{enumerate}
\item[(1')] 	for each $k\in K$, the $\infty$-category $X_k\cong p'(k)$ admits an initial object $t\in X_k$, and
\item[(2')] for each $p$-Cartesian edge $f:t'\to t$ over $p(f):k'\to k$ the object $t'\in X_{k'}$ is initial whenever $t\in X_k$ is,
\end{enumerate}

	Let $X'\subseteq X$ be the simplical subset spanned by those vertices $t\in X$ which are initial objects of $X_{p(t)}\cong p'\circ p(t)$. Then (as we shall show), every edge $f:t'\to t$ in $X'$ is $p$-Cartesian (when seen as an edge of $X$). To see this suppose that $f:t'\to t$ is such an edge. Let $f':t''\to t$ be a $p$-Cartesian edge in $X$ over $p(f)$; then (cf. \cite[Remark~2.4.1.4.]{Lurie:2009un}) there exists a 2-simplex $\sigma:\Delta^2\to X$ such that 
	$$\sigma(\Delta^{\{1,2\}})=f',\quad \sigma(\Delta^{\{0,2\}})=f,\quad\text{and}\quad p\big(\sigma(\Delta^{\{0,1\}})\big)=p(s_0(t')),$$
	where $s_0:K_0\to K_1$ is the degeneracy map.
	
	By assumption $t'',t'\in X_{p(s)}$ are both initial, and hence $\sigma(\Delta^{\{0,1\}})\in X_{p(t')}$ is an equivalence. In particular $\sigma(\Delta^{\{0,1\}})\in X$ is a $p$-Cartesian morphism. It follows from \cite[Proposition~2.4.1.7.]{Lurie:2009un} that $f=\sigma(\Delta^{\{0,2\}})$ is $p$-Cartesian.
	
	Now, by \cite[Proposition~2.4.4.9.]{Lurie:2009un}, there exists a section $t_\infty:K^{op}\to X'$; and by the previous discussion, $t_\infty$ is a Cartesian section.
	
	Now $\lim p'\cong \on{Fun}_{K^{op}}^{\on{Cart}}({K^{op}},X)$ is the $\infty$-category of Cartesian sections of $p$ (cf. \cite[Corollary~3.3.3.2]{Lurie:2009un}). Thus, we can identify $t_\infty$ with an element of $\lim p'$. We now claim that $t_\infty\in \lim p'$ is an initial object: Notice that $\lim p'$ is the full subcategory of $\on{Fun}_{K^{op}}(K^{op},X)$ spanned by the Cartesian sections.  Suppose we have a diagram
	$$\begin{tikzpicture}
		\mmat{m}{\partial \Delta^n & \on{Fun}_{K^{op}}^{\on{Cart}}({K^{op}},X) & \on{Fun}_{K^{op}}({K^{op}},X)\\ \Delta^n & &\\};
		\draw[->] (m-1-1) -- node {$f$} (m-1-2);
		\draw[right hook->] (m-1-2) -- (m-1-3);
		\draw[right hook->] (m-1-1) -- (m-2-1);
		\draw[->,dashed] (m-2-1) -- node[swap] {$\tilde f$} (m-1-3);
		\draw[->,dashed] (m-2-1) -- node {$\tilde f'$} (m-1-2);
	\end{tikzpicture}$$
	such that $f\rvert_{\{0\}}=t_\infty$. Then by \cite[Proposition~2.4.4.9.]{Lurie:2009un}, the arrow $\tilde f$ exists (making the diagram commute), but since $\on{Fun}_{K^{op}}^{\on{Cart}}({K^{op}},X)$ is the full subcategory of $\on{Fun}_{K^{op}}({K^{op}},X)$ spanned by the Cartesian sections, and all the vertices of $\tilde f$ lie in the image of $f$, and hence in $\on{Fun}_{K^{op}}^{\on{Cart}}({K^{op}},X)$, it follows that $\tilde f$ factors through a map $\tilde f':\Delta^n\to \on{Fun}_{K^{op}}^{\on{Cart}}({K^{op}},X)$.
	
	Thus $t_\infty\in \lim p'$ is an initial object.
	
	By construction, for every $k\in K$, $\pi_k(t_\infty)\in p'(k)$ is the image of the initial object $t_\infty(k)\in X_k$ under the equivalence of $\infty$-categories $X_k\xrightarrow{\cong} p'(k)$. Thus (2) follows from the uniqueness of initial objects. 
\end{proof}

Now, we can interpret arbitrary colimits in terms of initial objects using the concept of an undercategory, as follows: 
Suppose that $\mc{C}\in \Cat_\infty$ is an $\infty$-category and $q:I\to \mc{C}$ is a diagram of shape $I$ (where $I$ is a small simplicial set). Then a colimit diagram for $q$ is equivalent to an initial object of the undercategory, 
 $\mc{C}^{q/}$.

Suppose that $K$ is a (small) simplicial set and $\bar p':K^{\lhd}\to \Cat_\infty$ is a diagram with cone point $\mc{C}\in \Cat_\infty$. For any vertex $k\in K^\lhd$, let $\pi_k:\mc{C}\to p'(k)$ denote the corresponding functor in the cone.  Given any diagram $q:I\to \mc{C}$, since the formation of undercategories is natural, there exists a diagram $(\bar p')^{q/}:K^\lhd\to \Cat_\infty$ indexing the undercategories: $$(\bar p')^{q/}(k)\cong \big(\bar p'(k)\big)^{\pi_k\circ q/},$$ for all $k\in K^{\lhd}$, along with a natural transformation $(\bar p')^{q/}\to \bar p'$ which restricts at every $k\in K^\lhd$ to the canonical functors $\big(\bar p'(k)\big)^{\pi_k\circ q/}\to \bar p'(k)$.

\begin{Lemma}\label{lem:LimComUnder}
	If $\bar p':K^{\lhd}\to \Cat_\infty$ is a limit diagram, then so is $(\bar p')^{q/}:K^\lhd\to \Cat_\infty$.
\end{Lemma}
\begin{proof}
	Let $\bar p:\bar X \to K^\lhd$ be a coCartesian fibration classified by $\bar p'$; and let $\ast\in (K^{op})^\lhd$ denote the cone point. Since $\{\ast\}^\sharp\subseteq \big(K^\lhd\big)^\sharp$ is marked anodyne, we have a natural equivalence of the $\infty$-category of coCartesian sections:
	$$\on{Fun}_{K^\lhd}^{\on{coCart}}(K^\lhd,\bar X)\xrightarrow{\cong}\on{Fun}_{K^\lhd}^{\on{coCart}}(\{\ast\},\bar X)\cong \bar X\times_{K^\lhd}\{\ast\}\cong \mc{C}.$$
	Therefore, we may lift $q:I\to \mc{C}$ to a diagram $$\tilde q:I\to \on{Fun}_{K^\lhd}^{\on{coCart}}(K^\lhd,\bar X)$$ in a homotopically unique way. Choose a factorization $I \to I' \to  \on{Fun}_{K^\lhd}^{\on{coCart}}(K^\lhd,\bar X)$ of $\tilde q$, where $I \to I'$ is inner anodyne (and therefore a categorical equivalence) and $I' \to  \on{Fun}_{K^\lhd}^{\on{coCart}}(K^\lhd,\bar X)$ is an inner fibration (so that $I'$ is an $\infty$-category). The map $I \to I'$ is a categorical equivalence, and therefore cofinal. We are free to replace $I$ by $I'$, and may thereby assume that $I$ is an $\infty$-category.

Given two morphisms of simplicial sets $Y\to K^\lhd$, and $Z\to K^\lhd$, recall that $Z\diamond_{K^\lhd} Y$ denotes the relative (alternate) join of the simplicial sets $Z$ and $Y$, $$Z\diamond_{K^\lhd} Y:= Z\coprod_{Z\times_{K^\lhd} Y\times \{0\}}(Z\times_{K^\lhd} Y\times \Delta^1)\coprod_{Z\times_{K^\lhd} Y\times \{1\}}Y
$$ (cf. \cite[\S~4.2.2]{Lurie:2009un} and \cite{Joyal:2008wr}).

	Let $q_{K^\lhd}:I\times K^\lhd\to \bar X$ denote the composite
$$I\times K^\lhd\xrightarrow{q\times \on{Id}_{K^\lhd}}\on{Fun}_{K^\lhd}^{\on{coCart}}(K^\lhd,\bar X)\times K^\lhd\xrightarrow{\on{ev}} \bar X.$$
As in \cite[\S~4.2.2]{Lurie:2009un} we define $\bar X^{q_{K^\lhd}/}\to K^\lhd$ to be the simplicial set satisfying the universal property that
for any morphism of simplicial sets $Y\to K^\lhd$, 
 commutative diagrams of the form 
$$\begin{tikzpicture}
	\mmat{m}{I\times K^\lhd &\bar X\\(I\times K^\lhd)\diamond_{K^\lhd} Y& K^\lhd\\};
	\draw[->] (m-1-1) -- node {$q_{K^\lhd}$} (m-1-2);
	\draw[right hook->] (m-1-1) --  (m-2-1);
	\draw[->] (m-1-2) -- node {$\bar p$} (m-2-2);
	\draw[->] (m-2-1) -- (m-2-2);
	\draw[->] (m-2-1) --  (m-1-2);
\end{tikzpicture},$$
correspond to diagrams of the form 
$$\begin{tikzpicture}\node (y) at (-2,1) {$Y$};
\node (q) at (2,1) {$\bar X^{q_{K^\lhd}/}$};
\node (css) at (0,0) {$K^\lhd$};
\draw[->] (y) -- (q);
\draw[->] (y) -- (css);
\draw[->] (q) -- (css);
\end{tikzpicture}$$


	Note that $p\circ q_{K^\lhd}:I\times K^\lhd\to K^\lhd$ is just the projection, so $p\circ q_{K^\lhd}$ is a Cartesian fibration; and by \cite[Proposition~4.2.2.4.]{Lurie:2009un} $\bar X^{q_{K^\lhd}/}\to K^\lhd$ is a coCartesian fibration classified by $(\bar p')^{q/}:({K})^\lhd \to \Cat_\infty$. In particular, the fibre of $\bar X^{q_{K^\lhd}/}$ over any $k\in K^\lhd$ may be identified with the undercategory $\big(\bar p'(k)\big)^{\pi_k\circ q/}$ (cf. \cite[\S~4.2.2]{Lurie:2009un}).
	
	Let $X=\bar X\times_{K^\lhd} K$, and $q_{K}=q_{K^\lhd}\rvert_{I\times K}:I\times K\to X$. Then $\bar X^{q_{K^\lhd}/}\times_{K^\lhd} K\to K$ is canonically isomorphic to $X^{q_K/}\to K$. Consequently, by \cite[Proposition~3.3.3.1]{Lurie:2009un}, it suffices to show that whenever 
	\begin{subequations}
	\begin{equation}\label{eq:thetaequiv}\theta:\on{Fun}_{K^\lhd}^{\on{coCart}}(K^\lhd,\bar X)\to \on{Fun}_{K}^{\on{coCart}}(K,X)\end{equation} is an equivalence of $\infty$-categories, so is
	\begin{equation}\label{eq:thetaqequiv}\on{Fun}_{K^\lhd}^{\on{coCart}}(K^\lhd,\bar X^{q_{K^\lhd/}})\to \on{Fun}_{K}^{\on{coCart}}(K^\sharp,X^{q_K/}).\end{equation}\end{subequations}

	Using the identification $(I\times {K^\lhd})\diamond_{K^\lhd} (\Delta^n\times {K^\lhd})\cong (I\diamond \Delta^n)\times {K^\lhd}$, one sees that the $n$ simplices of $\on{Fun}_{K^\lhd}^{\on{coCart}}(K^\lhd,\bar X^{q_{K^\lhd/}})$ are lifting diagrams of the form 
$$\protect\tikz{
\node (m-1-1) at (0,2) {$I\times {K^\lhd}$};
\node (m-1-2) at (3,2) {$\bar X$};
\node (m-2-1) at (0,0) {$(I\diamond \Delta^n)\times {K^\lhd}$};
\node (m-2-2) at (3,0) {${K^\lhd}$};
	\draw[->] (m-1-1) -- node {$q_{K^\lhd}$} (m-1-2);
	\draw[right hook->] (m-1-1) --  (m-2-1);
	\draw[->] (m-1-2) -- node {$p$} (m-2-2);
	\draw[->] (m-2-1) -- (m-2-2);
	\draw[dashed,->] (m-2-1) -- node {$\sigma$}  (m-1-2);
}$$
such that for each vertex $v$ of $I\diamond \Delta^n$, the restriction $\sigma\rvert_{\{v\}\times {K^\lhd}}:{K^\lhd}\to \bar X$ is coCartesian. Thus, $\on{Fun}_{K^\lhd}^{\on{coCart}}(K^\lhd,\bar X^{q_{K^\lhd/}})\cong \big(\on{Fun}_{K^\lhd}^{\on{coCart}}({K^\lhd},\bar X)\big)^{\tilde q/}$.

Similarly, $\on{Fun}_K^{\on{coCart}}K,X^{q_{K/}})\cong \big(\on{Fun}_K^{\on{coCart}}({K}, X)\big)^{\theta\circ \tilde q/}$. It follows that \eqref{eq:thetaqequiv} is an equivalence whenever \eqref{eq:thetaequiv} is.
\end{proof}

Combining Proposition~\ref{lem:LimOfInit} and Lemma~\ref{lem:LimComUnder} yields the general case:
\begin{Theorem}\label{thm:LimInLim}Let $I$ and $K$ be small simplicial sets, and suppose $p':K\to \Cat_\infty$ is a functor. Let $\lim p'\in \Cat_\infty$ denote the limit $\infty$-category and for each $k\in K$, let $\pi_k:\lim p'\to p'(k)$ denote the corresponding functor in the limit cone.
Suppose that $q:I\to \lim p'$ is a diagram indexed by $I$, and that
\begin{enumerate}
	\item for each vertex $k\in K$, the composite diagram $\pi_k\circ q:I\to p'(k)$ has a (co)limit diagram, and
	\item for each edge $f:k\to k'$ of $K$, the functor $p'(f):p'(k)\to p'(k')$ takes (co)limit diagrams extending $\pi_k\circ q$ to (co)limit diagrams extending $\pi_{k'}\circ q$.
\end{enumerate} 
Then:
\begin{description}
	\item[A] there exists a map $\bar q:I^\rhd\to \lim p'$ which extends $q$ and such that each composite $\pi_{k'}\circ \bar q:I^\rhd\to p'(k)$ is a (co)limit (co)cone, and 
	\item[B] an arbitrary extension $\bar q:I^\rhd\to \lim p'$ of $q$ is a (co)limit diagram extending $q$ if and only if each composite $\pi_{k'}\circ \bar q:I^\rhd\to p'(k)$ is a (co)limit diagram extending $\pi_{k'}\circ q$.
\end{description} 

In particular, if
\begin{enumerate}
	\item[(1')] for each vertex $k\in K$ the  $\infty$-category $p'(k)$ admits (co)limits for all diagrams indexed by $I$, and
	\item[(2')] for each edge $f:k\to k'$ of $K$, the functor $p'(f):p'(k)\to p'(k')$ preserves (co)limits for all diagrams indexed by $I$.
\end{enumerate}
then the limit $\infty$-category $\lim p'$ admits all (co)limits of shape $I$, and the functors $\pi_k:\lim p'\to p'(k)$ fitting into the limit cone preserve all (co)limits of shape $I$.
%
\end{Theorem}

\begin{proof}

Let $\bar p':K^\lhd\to \Cat_\infty$ be a limit cone extending $p'$ which maps the cone point $\infty\in K^\lhd$ to $\lim p'$ and the cone edge $\infty\to k$ to $\pi_k$ for each $k\in K^\lhd$. Let $(\bar p')^{q/}:K^\lhd\to \Cat_\infty$ denote corresponding diagram of undercategories, as in Lemma~\ref{lem:LimComUnder}. Then $(\bar p')^{q/}:K^\lhd\to \Cat_\infty$ is a limit cone which (by assumptions (1) and (2)) satisfies the assumptions for Proposition~\ref{lem:LimOfInit}. Now for any $k\in K^\lhd$, the $\infty$-category of diagrams $I^\rhd\to \bar p'(k)$ extending $\pi_k\circ q$ is equivalent to the undercategory $\big(\bar p'(k)\big)^{\pi_k\circ q}$, and this equivalence identifies colimit diagrams with initial objects of the undercategory. Moreover, for any edge $k\to k'$ in $K^\lhd$, the functor $\big(\bar p'(k)\big)^{\pi_k\circ q/}\to \big(\bar p'(k')\big)^{\pi_{k'}\circ q/}$ preserves initial objects if and only if the functor
 $\bar p'(k)\to \bar p'(k')$ takes colimits diagrams extending $\pi_k\circ q$ to colimit diagrams extending $\pi_{k'}\circ q$.

Therefore statements \textbf{A} and \textbf{B} follow from Proposition~\ref{lem:LimOfInit}.

\end{proof}

 \subsection{Adjunctions and Kan Extensions}
 \subsubsection{Limits of adjunctions.} Our first application of Theorem~\ref{thm:LimInLim} is to prove the following result:
\begin{Theorem}[A limit of adjunctions is an adjunction]\label{thm:limadj} Suppose 
	$$F_k:\mc{C}_k \rightleftarrows\mc{D}_k:G_k, \quad k\in K$$ is diagram of adjunctions 
	coherently indexed by a small simplical set $K$, i.e. given by a diagram $(f\dashv g):K\to \on{Adj}$ into the $\infty$-category of adjunctions.
	 Let 
	$\mc{C}=\lim_{k\in K}\mc{C}_k$ and  $\mc{D}=\lim_{k\in K}\mc{D}_k,$
	and suppose that
	\begin{itemize}
	\item 	
 there is a functor $F:\mc{C}\to\mc{D}$ which fits into the cone edge of a diagram $\bar f:K^\lhd\times \Delta^1\to \Cat_\infty$ extending $f:K\times \Delta^1\to \Cat_\infty$, and
 \item  there is a functor $G:\mc{D}\to \mc{C}$ which fits into the cone edge of a diagram $\bar g:K^\lhd\times \Delta^1\to \Cat_\infty$,  extending $g:K\times \Delta^1\to \Cat_\infty$,
	\end{itemize}
	such that
	\begin{enumerate}
	\item  the restrictions $\bar f\rvert_{K^\lhd\times\{0\}}$ and  $\bar g\rvert_{K^\lhd\times\{1\}}$ are limit cones for $\mc{C}$, and
	\item  the restrictions $\bar f\rvert_{K^\lhd\times\{1\}}$ and  $\bar g\rvert_{K^\lhd\times\{0\}}$ are limit cones for $\mc{D}$.
	\end{enumerate}

	Then $F$ and $G$ form a pair of adjoint functors
	$$F:\mc{C} \rightleftarrows\mc{D}:G$$
\end{Theorem}
 We defer the proof until later: we will first need to give a precise definition of the $\infty$-category, $\on{Adj}$, of adjunctions. To do this, we will use the framework for adjunctions of $\infty$-categories in terms of \emph{pairing of $\infty$-categories}, as developed in \cite{Lurie:ZsNVU6Ru}. For now we give an immediate corollary:
 
\begin{Corollary}\label{cor:KanClosed}
	Suppose that $\delta:I\to I'$ is a morphism of simplicial sets. Let $\Cat_\infty^\delta\subset\Cat_\infty$ be the subcategory consisting of
	\begin{itemize}
 \item those $\infty$-categories $\mc{C}$ which admit left (right) Kan extensions along $\delta$ for any functor $f:I\to \mc{C}$, and 
 \item those functors $\mc{C}\to\mc{C}'$ which preserve left (right) Kan extensions along $\delta$.
	\end{itemize}
	
	Then $\Cat_\infty^\delta\subset\Cat_\infty$ is closed under (small) limits.
\end{Corollary}
\begin{proof}
Let $\phi:\Cat_\infty\to\on{Fun}(\Delta^1,\Cat_\infty)$ be the functor 	which sends an $\infty$-category $\mc{C}$ to the pullback-functor
\begin{equation}\label{eq:rKanExt}\on{Fun}(I,\mc{C})\leftarrow \on{Fun}(I',\mc{C}):\delta^*.\end{equation}
Evaluating $\phi$ at either endpoint of $\Delta^1$
\begin{align*}
\on{ev}_{\{0\}}\circ \phi:\Cat_\infty\xrightarrow{\mc{C}\mapsto \on{Fun}(I',\mc{C})}\Cat_\infty\\
\on{ev}_{\{1\}}\circ \phi:\Cat_\infty\xrightarrow{\mc{C}\mapsto \on{Fun}(I,\mc{C})}\Cat_\infty
\end{align*}
results in continuous functors (they are right adjoints). Therefore \cite[Corollary~5.1.2.3]{Lurie:2009un} implies that $\phi$ is continuous.

Now $\mc{C}$ admits left Kan extensions along $\delta$ if and only if \eqref{eq:rKanExt} is a right adjoint. In particular, we
 may identify $\Cat_\infty^\delta$ with the pullback 
$$\begin{tikzpicture}
	\mmat{m}{\Cat_\infty^\delta& \Cat_\infty\\\on{Adj}&\on{Fun}(\Delta^1,\Cat_\infty)\\};
	\draw[->] (m-1-2) -- node {$\phi$} (m-2-2);
	\draw[->] (m-2-1) -- node {$R$} (m-2-2);
	\draw[dashed,->] (m-1-1) -- (m-1-2);
	\draw[dashed,->] (m-1-1) -- (m-2-1);
\end{tikzpicture}$$
Where $R:\on{Adj}\to\on{Fun}(\Delta^1,\Cat_\infty)$ is the functor which sends an adjunction $(F\dashv G)$ to its right adjoint $G$.
By Theorem~\ref{thm:limadj}, the functor $R$ is continuous. Therefore, by Theorem~\ref{thm:LimInLim}, $\Cat_\infty^\delta$ admits all small limits and the functor $\Cat_\infty^\delta\to \Cat_\infty$ is continuous.
\end{proof}

 
 \subsubsection{Pairings of $\infty$-categories} We recall the theory of \emph{pairings of $\infty$-categories}; essentially all this material is taken from \cite{Lurie:ZsNVU6Ru}, though we provide proofs for certain details that will be important to us when discussing adjunctions. Recall that the $\infty$-category of pairings
 $$\on{CPair}\subseteq \on{Fun}(\Lambda_0^2,\Cat_\infty)$$
  is the full subcategory consisting of diagrams $$\begin{tikzpicture}
 \node  (m-1-1)  at (0,1) {$\mc{M}$};
 \node  (m-1-2)  at (1,0) {$\mc{D}^{op}$};
 \node  (m-2-1)  at (-1,0) {$\mc{C}$};
\draw[->] (m-1-1) -- node {$\lambda_D$} (m-1-2);
\draw[->] (m-1-1) -- node [swap] {$\lambda_C$} (m-2-1);
\end{tikzpicture}$$
such that $\lambda:\mc{M}\xrightarrow{\lambda_C\times\lambda_D} \mc{C}\times\mc{D}^{op}$ is equivalent to a right fibration. 

 
 Given such a right fibration, $\lambda$ is classified by a functor (cf. \cite[\S~2.2.1]{Lurie:2009un}) $$\mc{D}\times\mc{C}^{op}\to \Spc$$ to the $\infty$-category, $\Spc$, of spaces; or equivalently a functor \begin{equation}\label{eq:lD'}\lambda':\mc{D}\to \on{Fun}(\mc{C}^{op},\Spc)=:\Pre(\mc{C})\end{equation}
 to the $\infty$-category of presheaves over $\mc{C}$. 
  Here $\lambda'$ takes each vertex $d\in \mc{D}$ to the right fibration \begin{equation}\label{eq:Pred}\mc{M}\times_{\mc{D}^{op}}\{d\}\to \mc{C}.\end{equation}


   As in \cite{Lurie:ZsNVU6Ru}, we call an object of $m\in\mc{M}$ \emph{right universal} if it is a terminal object of $\mc{M}\times_{\mc{D}^{op}}\{\lambda_D(m)\}$
and we call a right fibration
  \begin{equation}\label{eq:RFibrPair}\lambda:\mc{M}\to \mc{C}\times\mc{D}^{op}\end{equation}
   a \emph{right representable pairing}, if for each $d\in\mc{D}^{op}$, there exists a right universal object in the fibre $\mc{M}\times_{\mc{D}^{op}}\{d\}$ over $d$. In this case, for each $d\in\mc{D}$, the right fibration \eqref{eq:Pred} is representable (cf. \cite[Proposition 4.4.4.5]{Lurie:2009un}). The Yoneda embedding\footnote{following Theodore Johnson-Freyd's suggestion, we denote the Yoneda embedding by the first Hiragana character of his name, \yoneda.} $\yon:\mc{C}\to \Pre(\mc{C})$ identifies $\mc{C}$ with the full subcategory of $\Pre(\mc{C})$ spanned by the representable presheaves (cf \cite[Proposition~5.1.3.1]{Lurie:2009un}); whence it follows that $\lambda'$ factors through $\mc{C}$,
  $$\begin{tikzpicture}
  	\mmat{m}{\mc{D}&&\mc{C}\\&\Pre(\mc{C})&\\};
  	\draw[->] (m-1-1) -- node[swap] {$\lambda'$} (m-2-2);
  	\draw[right hook->] (m-1-3) -- node {$\yon$} (m-2-2);
  	\draw[dashed,->] (m-1-1) -- node {$\lambda^R$} (m-1-3);
  \end{tikzpicture}$$
  The Yoneda lemma implies that we have a weak equivalence of spaces $$\on{Hom}_\mc{C}\big(c,\lambda^R(d)\big)\cong \{c\}\times_{\mc{C}}\mc{M}\times_{\mc{D}^{op}}\{d\},$$
  which depend naturally on $(c,d)\in \mc{C}^{op}\times\mc{D}$.

  Similarly, an object of $m\in\mc{M}$ is called \emph{left universal} if it is a terminal object of $\mc{M}\times_{\mc{C}}\{\lambda_C(m)\}$, and the
 right fibration \eqref{eq:RFibrPair} is called
  a \emph{left representable pairing}, if for each $c\in\mc{C}$, there exists a left universal object in the fibre $\mc{M}\times_\mc{C}\{c\}$ over $c$. As before, this determines a functor $\lambda^L:\mc{C}\to \mc{D}$; and the yoneda Lemma implies that we have weak equivalences of spaces
  $$\on{Hom}_\mc{C}\big(\lambda^L(c),d\big)\cong \{c\}\times_{\mc{C}}\mc{M}\times_{\mc{D}^{op}}\{d\}\cong\on{Hom}_\mc{C}\big(c,\lambda^R(d)\big)$$
  depending naturally on  $(c,d)\in \mc{C}^{op}\times\mc{D}$. Indeed, $\lambda^R$ is a right adjoint to $\lambda^L$ (cf. \cite{Lurie:ZsNVU6Ru} or \cite[\S~5.2.6]{Lurie:2009un} for more details).

  Suppose that $\mc{M}\to \mc{C}\times\mc{D}^{op}$ and $\mc{M}'\to \mc{C}'\times\mc{D}'^{op}$ are two right representable right fibrations of $\infty$-categories, then a morphism of diagrams
  $$\begin{tikzpicture}
  	\mmat{m}{\mc{M}&\mc{M}'\\\mc{C}\times\mc{D}^{op}&\mc{C}'\times\mc{D}'^{op}\\};
  	\draw[->] (m-1-1) -- (m-2-1);
  	\draw[->] (m-1-2) -- (m-2-2);
  	\draw[dashed,->] (m-1-1) -- node {$\gamma$} (m-1-2);
  	\draw[dashed,->] (m-2-1) -- node {$\alpha\times\beta$} (m-2-2);
  \end{tikzpicture}$$
  is called \emph{right representable} if it takes right universal objects to right universal objects.
  
The $\infty$-category of right-representable pairings $\on{CPair}^R\subseteq\on{CPair}$ is defined to be the subcategory whose objects are equivalent to right representable pairings, and whose morphisms are equivalent to right representable morphisms. The $\infty$-category of left-representable pairings $\on{CPair}^L\subseteq\on{CPair}$ is defined analogously. 

 \begin{Lemma}\label{lem:CPairClosed}
Let $\on{CPair}^R\subseteq\on{CPair}\subseteq \on{Fun}(\Lambda_0^2,\Cat_\infty)$ be the $\infty$-categories defined in \cite{Lurie:ZsNVU6Ru}. Then both subcategories are closed under small limits. 
\end{Lemma}
\begin{proof}
Since $\on{CPair}\subseteq \on{Fun}(\Lambda_0^2,\Cat_\infty)$ is a reflective localization (cf. \cite[Remark~4.2.9]{Lurie:ZsNVU6Ru}), it is closed under small limits; so we need only show that $\on{CPair}^R\subseteq\on{CPair}$ is also closed under small limits. 

Let $p:K^{op}\to \on{CPair}^R$ be a diagram (for which we wish to compute the limit). 
The composite functor $K^{op}\to \on{CPair}^R\to \on{CPair}$ is classified by a diagram of simplicial sets 
$$\begin{tikzpicture}
\mmat{m}{\mc{M}_p	& \mc{D}_p^{op}\\ \mc{C}_p & K^{op}\\};
\draw[->] (m-1-1) -- node {$\lambda_C$} (m-1-2);
\draw[->] (m-1-1) -- node [swap] {$\lambda_D$} (m-2-1);
\draw[->] (m-1-2) -- node {$\tilde p_D$} (m-2-2);
\draw[->] (m-2-1) -- node [swap]{$\tilde p_C$} (m-2-2);
\end{tikzpicture}$$
where $\lambda_p=\lambda_C\times_K\lambda_D:\mc{M}_p\to \mc{C}_p\times_K\mc{D}_p$ is a right fibration and $\tilde p_C$ and $\tilde p_D$ are Cartesian fibrations. 
The limit of $p$ is a right fibration (cf. \cite[Remark~4.2.9.]{Lurie:ZsNVU6Ru}) 
\begin{equation}\label{eq:LimPair}\mc{M}:=\lim p_M\to \lim p_C\times \lim p_D=:\mc{C}\times\mc{D},\end{equation}
 where $p_M:K\to \Cat_\infty$ is the functor classified by $(\tilde p_C\times_K\tilde p_D)\circ\lambda_p$, and $p_C$ and $p_D$ are classified by $\tilde p_C$ and $\tilde p_D$ respectively. We need to show that \eqref{eq:LimPair} is right representable and that the canonical morphisms to \eqref{eq:LimPair} are right representable.

Now a vertex $d\in \lim p_D=\mc{D}$ can be identified with a Cartesian section $\tilde d:K\to \mc{D}_p$ of $\tilde p_D$ (cf. \cite[Corollary~3.3.3.2.]{Lurie:2009un}). Let $\mc{M}_{p,d}\xrightarrow{\lambda_d} K$ be the Cartesian fibration which fits into the pullback square
$$\begin{tikzpicture}
	\mmat{m}{\mc{M}_{p,d}& \mc{M}_p\\K& \mc{D}_p\\};
	\draw[->] (m-1-1) -- (m-1-2);
	\draw[->] (m-1-1) -- (m-2-1);
	\draw[->] (m-1-2) -- (m-2-2);
	\draw[->] (m-2-1) -- node {$\tilde d$} (m-2-2);
\end{tikzpicture}$$
and let $q:K^{op}\to \Cat_\infty$ be the corresponding functor. Then $\lim q\cong \mc{M}_d:=\mc{M}\times_{\mc{D}}\{d\}$ (since taking pullbacks commutes with taking limits). To show that \eqref{eq:LimPair} is right representable, we need to show that $\mc{M}_d$ has a final object. However, since $p$ takes values in the $\infty$-category $\on{CPair}^R$ of right representable pairings, for each $k\in K$, the pullback $$\mc{M}_{p,d,k}:=\mc{M}_{p,d}\times_{K}\{k\}$$ has a final object, and for each morphisms $(k\to k')$ in $K^{op}$, the corresponding functor $\mc{M}_{p,d,k}\to \mc{M}_{p,d,k'}$ takes final objects to final objects. Thus, by \cite[Theorem.~3.16.]{Riehl:2014ut} (or Proposition~\ref{thm:LimInLim}), the limit $\mc{M}_d\cong \lim q$ has a final object, and moreover the canonical morphisms $\mc{M}_d\to \mc{M}_{p,d,k}$ preserve final objects.

It follows that $\lim p\in \on{CPair}$ is in fact an element of $\on{CPair}^R$ and that the limit cone is a diagram in $\on{CPair}^R$; i.e. $\on{CPair}^R\subseteq \on{CPair}$ is closed under small limits.

\end{proof}

\begin{Proposition}\label{prop:FunInfCat}
There are equivalences of $\infty$-categories 
\begin{align}
\on{CPair}^L&\cong \on{Fun}(\Delta^1,\Cat_\infty),\\
\on{CPair}^R&\cong \on{Fun}(\Delta^1,\Cat_\infty),
\end{align}
which associate a left representable pairing 	
$\lambda:\mc{M}\to \mc{C}\times\mc{D}^{op}$ to the functor $\lambda^L:\mc{C}\to \mc{D}$, and a right representable pairing  $\lambda:\mc{M}\to \mc{C}\times\mc{D}^{op}$ to the functor $\lambda^R:\mc{D}\to \mc{C}$.
\end{Proposition}
\begin{proof}
As in \cite{Lurie:ZsNVU6Ru}, we say that a right fibration \eqref{eq:RFibrPair} is a perfect pairing if it is both left and right representable, and an object $m\in \mc{M}$ is left universal if and only if it is right universal.
	Let $\on{CPair}^{\textrm{perf}}\subset\on{CPair}^L$ be the full subcategory spanned by the perfect pairings. Let $\phi:\on{CPair}^{\textrm{perf}}\to\Cat_\infty$ denote the forgetful functor which sends a perfect pairing \eqref{eq:RFibrPair} to $\mc{C}$; and let $\widetilde{\on{Fun}(\Delta^1,\Cat_\infty)}$ denote the $\infty$-category fitting into the pullback square
	$$\begin{tikzpicture}
		\mmat{m}{\widetilde{\on{Fun}(\Delta^1,\Cat_\infty)}&\on{CPair}^{\textrm{perf}}\\\on{Fun}(\Delta^1,\Cat_\infty)&\Cat_\infty\\};
  	\draw[dashed,->] (m-1-1) -- (m-2-1);
  	\draw[->] (m-1-2) -- node {$\phi$} (m-2-2);
  	\draw[dashed,->] (m-1-1) -- (m-1-2);
  	\draw[->] (m-2-1) -- node {$\on{ev}_1$} (m-2-2);
	\end{tikzpicture}$$
	Note: since the bottom arrow is a Cartesian fibration (cf. \cite[Corollary~2.4.7.11]{Lurie:2009un}), this homotopy pullback can be computed as a pullback of simplicial sets (cf. \cite[Corollary 3.3.1.4]{Lurie:2009un}).
	Since $\phi$ is an equivalence of $\infty$-categories (cf. \cite[Remark~4.2.12]{Lurie:ZsNVU6Ru}), the left arrow defines an equivalence between $\widetilde{\on{Fun}(\Delta^1,\Cat_\infty)}$ and $\on{Fun}(\Delta^1,\Cat_\infty)$.
	
	The inclusion $\on{CPair}^{\textrm{perf}}\subset \on{Fun}(\Lambda_0^2,\Cat_\infty)$ allows us to identify $\widetilde{\on{Fun}(\Delta^1,\Cat_\infty)}$ with diagrams of the form 
	\begin{equation}\label{eq:tilfun}\begin{tikzpicture}
	\node (c) at (-1,1)  {$\mc{C}$};
	\node (d) at (0,0)  {$\mc{D}$};
	\node (p) at (1,1)  {$\mc{P}$};
	\node (d1) at (2,0)  {$\tilde{\mc{D}}^{op}$};
	\draw[->] (c)-- node {$f$} (d);
	\draw[->] (p)-- (d);
	\draw[->] (p)-- (d1);
\end{tikzpicture}\end{equation}
where $\mc{P}\to \mc{D}\times\tilde{\mc{D}}^{op}$ is a perfect pairing. Taking the limit of such a diagram yields
	$$\begin{tikzpicture}
	\node (m) at (1,2) {$\mc{M}$};
	\node (c) at (-1,1)  {$\mc{C}$};
	\node (d) at (0,0)  {$\mc{D}$};
	\node (p) at (1,1)  {$\mc{P}$};
	\node (d1) at (2,0)  {$\tilde{\mc{D}}^{op}$};
	\draw[->] (c)-- (d);
	\draw[->] (p)-- (d);
	\draw[->] (p)-- (d1);
	\draw[dashed,->] (m) -- (c);
	\draw[dashed,->] (m) -- (d);
	\draw[dashed,->] (m) -- (p);
	\draw[dashed,->] (m) -- (d1);
\end{tikzpicture}$$
where \begin{equation}\label{eq:lam'}\lambda:\mc{M}\to \mc{C}\times\tilde{\mc{D}}^{op}\end{equation} is a left-representable pairing. Thus we get a functor $$A:\widetilde{\on{Fun}(\Delta^1,\Cat_\infty)}\to \on{CPair}^L,$$
sending an object of the form \eqref{eq:tilfun} to the left representable pairing \eqref{eq:lam'}. Notice that, by construction, the functor $f$ appearing in \eqref{eq:tilfun} is equivalent to $\lambda^L:\mc{C}\to \tilde{\mc{D}}$. 

It remains to show that $A$ is an equivalence of categories. The essential surjectivity of $A$ is explained in\cite[Remark~4.2.13]{Lurie:ZsNVU6Ru}.
We argue that $A$ is fully faithful: Suppose that 
	$$\begin{tikzpicture}
	\node at (-2,.5) {$\tilde{f}=$};
	\node (m) at (1,2) {$\mc{M}$};
	\node (c) at (-1,1)  {$\mc{C}$};
	\node (d) at (0,0)  {$\mc{D}$};
	\node (p) at (1,1)  {$\mc{P}$};
	\node (d1) at (2,0)  {$\tilde{\mc{D}}^{op}$};
	\draw[->] (c)-- node {$f$} (d);
	\draw[->] (p)-- (d);
	\draw[->] (p)-- (d1);
	\draw[dashed,->] (m) -- (c);
	\draw[dashed,->] (m) -- (d);
	\draw[dashed,->] (m) -- (p);
	\draw[dashed,->] (m) -- (d1);
	\node at (3,.5) {and};
	\node at (4,.5) {$\tilde{f}'=$};
	\node (m') at (7,2) {$\mc{M}'$};
	\node (c') at (5,1)  {$\mc{C}'$};
	\node (d') at (6,0)  {$\mc{D}'$};
	\node (p') at (7,1)  {$\mc{P}'$};
	\node (d1') at (8,0)  {$\tilde{\mc{D}}'^{op}$};
	\draw[->] (c')-- node {$f'$} (d');
	\draw[->] (p')-- (d');
	\draw[->] (p')-- (d1');
	\draw[dashed,->] (m') -- (c');
	\draw[dashed,->] (m') -- (d');
	\draw[dashed,->] (m') -- (p');
	\draw[dashed,->] (m') -- (d1');
\end{tikzpicture}$$
are a pair of objects in $\widetilde{\on{Fun}(\Delta^1,\Cat_\infty)}$, with 
$$A(\tilde{f})=\big(\lambda:\mc{M}\to \mc{C}\times\tilde{\mc{D}}^{op}\big),\text{ and }A(\tilde{f}')=\big(\lambda':\mc{M}'\to \mc{C}'\times\tilde{\mc{D}}'^{op}\big).$$
We need to show that the natural map between the mapping spaces
\begin{equation}\label{eq:AFullFaith}A:\on{Map}_{\widetilde{\on{Fun}(\Delta^1,\Cat_\infty)}}(\tilde f,\tilde f')\to \on{Map}_{\on{CPair}^L}(\mc{M},\mc{M}')\end{equation}
is a homotopy equivalence.\footnote{For two left-representable pairings $\lambda:\mc{M}\to \mc{C}\times\tilde{\mc{D}}^{op}$ and $\lambda':\mc{M}'\to \mc{C}'\times\tilde{\mc{D}}'^{op}$, the mapping space
$\on{Map}_{\on{CPair}^L}(\mc{M},\mc{M}')$ is the subspace of  $$\on{Map}_{\Cat_\infty}(\mc{C},\mc{C}')\times^h_{\on{Map}_{\Cat_\infty}(\mc{M},\mc{C}')}\on{Map}_{\Cat_\infty}(\mc{M},\mc{M}')\times^h_{\on{Map}_{\Cat_\infty}(\mc{M},\tilde{\mc{D}}')}\on{Map}_{\Cat_\infty}(\tilde{\mc{D}},\tilde{\mc{D}}')$$ which preserves left universal objects. Notice that the homotopy pullbacks can be taken to be strict pullbacks when $\lambda$ and $\lambda'$ are right fibrations.}

On the one hand, 
$$\on{Map}_{\widetilde{\on{Fun}(\Delta^1,\Cat_\infty)}}(\tilde f,\tilde f')
\cong \on{Map}_{\Cat_\infty}(\mc{C},\mc{C}')\times^h_{\on{Map}_{\Cat_\infty}(\mc{C},\mc{D}')}\on{Map}_{\on{CPair}^L}(\mc{P},\mc{P}')$$
But, since $\mc{P}'\to \mc{D}'\times\tilde{\mc{D}}'^{op}$ is a perfect pairing, \cite[Proposition~4.2.10]{Lurie:ZsNVU6Ru} shows we have homotopy equivalences of mapping spaces 
$$\on{Map}_{\on{CPair}^L}(\mc{P},\mc{P}')\xrightarrow{\cong}\on{Map}_{\Cat_\infty}(\tilde{\mc{D}},\tilde{\mc{D}}')\xleftarrow{\cong}\on{Map}_{\on{CPair}^L}(\mc{M},\mc{P}').$$ Consequently, 
\begin{subequations}\label{eq:HomEquivPerfPair}\begin{equation}\on{Map}_{\widetilde{\on{Fun}(\Delta^1,\Cat_\infty)}}(\tilde f,\tilde f')
\cong \on{Map}_{\Cat_\infty}(\mc{C},\mc{C}')\times^h_{\on{Map}_{\Cat_\infty}(\mc{C},\mc{D}')}\on{Map}_{\on{CPair}^L}(\mc{M},\mc{P}')\end{equation}

 On the other hand, since $\mc{M}'$ is a pullback of $\mc{P}'$, we have a homotopy equivalence of mapping spaces \begin{equation}\on{Map}_{\on{CPair}^L}(\mc{M},\mc{M}')\cong \on{Map}_{\Cat_\infty}(\mc{C},\mc{C}')\times^h_{\on{Map}_{\Cat_\infty}(\mc{C},\mc{D}')}\on{Map}_{\on{CPair}^L}(\mc{M},\mc{P}').\end{equation}\end{subequations}
 It follows from \eqref{eq:HomEquivPerfPair} that \eqref{eq:AFullFaith} is a homotopy equivalence. In particular, $A$ is fully faithful.
%

\end{proof}
 

\subsubsection{The $\infty$-category of Adjunctions, and the proof of Theorem~\ref{thm:limadj}}
We are now in a position to define the $\infty$-category of adjunctions and to prove Theorem~\ref{thm:limadj}.

\begin{Definition}
	The $\infty$-category of adjunctions, $\on{Adj}$ is defined as the pullback of $\infty$-categories
\begin{equation}\label{eq:AdjPull}\begin{tikzpicture}
	\node (adj) at (-7,0) {$\on{Adj}$};
	\node (fun1L) at (-4,-1) {$\on{Fun}(\Delta^1,\Cat_\infty)$};
	\node (fun1R) at (-4,1) {$\on{Fun}(\Delta^1,\Cat_\infty)$};
	\node (cpL) at (-1,-1) {$\on{CPair}^L$};
	\node (cpR) at (-1,1) {$\on{CPair}^R$};
	\node (cp) at (1,0) {$\on{CPair}$};
	\draw[->] (fun1L) -- node {$\cong$} (cpL);
	\draw[->] (fun1R) -- node {$\cong$} (cpR);
	\draw[->] (cpL) -- (cp);
	\draw[->] (cpR) -- (cp);
	\draw[dashed,->] (adj) -- (fun1L);
	\draw[dashed,->] (adj) -- (fun1R);
\end{tikzpicture}\end{equation}
\end{Definition}

\begin{proof}[Proof of Theorem~\ref{thm:limadj}]
By 	Lemma~\ref{lem:CPairClosed} and Proposition~\ref{thm:LimInLim}, each of the categories in the diagram \eqref{eq:AdjPull}
are complete, and each of the functors in the diagram preserve small limits. Thus, any diagram $(f\dashv g):K\to \on{Adj}$ admits a limit. Moreover, the limiting left adjoint $F:\mc{C}\to \mc{D}$ is a limiting functor for the diagram $f:K\to  \on{Fun}(\Delta^1,\Cat_\infty)$. In particular, $F$ can be characterized as in the statement of the theorem (cf. \cite[Corollary~5.1.2.3]{Lurie:2009un}). 

Similarly, the limiting right adjoint $G:\mc{D}\to\mc{C}$ is a limiting functor for the diagram $g:K\to  \on{Fun}(\Delta^1,\Cat_\infty)$;  so $G$ can be characterized as in the statement of the theorem.
\end{proof}

\section{Complete $k$-fold Segal objects.}
Let $\BDelta$ denote the simplex category, and for any simplicial set $K$, let $\BDelta_{/K}\to \BDelta$ denote the corresponding category of simplices\footnote{The objects of $\BDelta_{/K}$ over $[n]\in \BDelta$ are simplicial maps $\Delta^n\to K$ from the standard $n$-simplex, and morphisms in $\BDelta_{/K}$ over a morphism $f:[m]\to [n]$ are commutative diagrams
$$\protect\tikz{
	\node (m) at (-1,1) {$\Delta^m$};
	\node (n) at (1,1) {$\Delta^n$};
	\node (k) at (0,0) {$K$};
	\draw[->] (m) -- node {$f$} (n);
	\draw[->] (m) -- (k);
	\draw[->] (n) -- (k);
}$$
Equivalently, $\big(\BDelta_{/K}\to \BDelta\big)=\big(\int^{\BDelta} K\to \BDelta\big)$ is the Grothendieck fibration (or category of elements) associated the functor $K:\BDelta^{op}\to \on{Sets}$.} of $K$.
The spine of the standard $n$-simplex is the subsimplicial set
$$\on{Sp}(n)=\overset{n}{\overbrace{\Delta^{\{0,1\}}\coprod_{\Delta^{\{1\}}}\cdots \coprod_{\Delta^{\{n-1\}}}\Delta^{\{n-1,n\}}}}\subseteq \Delta^n.$$
generated by the 1-simplices $\Delta^{\{i,i+1\}}\subseteq\Delta^n$. The inclusion $\on{Sp}(n)\subseteq \Delta^n$ is a categorical equivalence,\footnote{In fact, the model structure on simplicial sets for $\infty$-category is the Cisinski model structure induced by the localizer which consists of the inclusions $\on{Sp}(n)\subseteq \Delta^n$ (cf. \cite{Ara:2012uj}).} and a simplical object $X_\bullet:\BDelta\to \mc{X}$ in an $\infty$-category $\mc{X}$ is called a \emph{category object} if it satisfies the so-called \emph{Segal conditions} (cf. \cite{Rezk:2001ey}): i.e. for each $n\geq 0$, the natural map
 \begin{equation}\label{eq:Segal}X_n\to \lim_{\BDelta^{op}_{/\on{Sp}(n)}} X_\bullet \cong \overset{n}{\overbrace{X_1\times_{X_0}\cdots \times_{X_0}X_1}}\end{equation}
 is an equivalence. 
 
 Given a category object $X_\bullet$ in $\mc{X}$, one should think of $X_0\in \mc{X}$ as describing the objects of an $(\infty,1)$-category internal to $\mc{X}$, $X_1\in \mc{X}$ as describing the morphisms of an $(\infty,1)$-category internal to $\mc{X}$, $X_i\in\mc{X}$ as describing the object classifying composable $i$-tuples of morphisms, and the various structural maps between the $X_i$'s as describing the homotopy-associative composition and units.  
 
 Now suppose that $\mc{X}$ is an $\infty$-topos. We let $\on{Cat}(\mc{X})\subseteq \on{Fun}(\BDelta^{op},\mc{X})$ denote the full subcategory spanned by the category objects. Unfortunately, $\on{Cat}(\mc{X})$ doesn't describe the correct homotopy theory of $(\infty,1)$-categories internal to $\mc{X}$; one must localize with respect to an appropriate class of ``fully faithful and essentially surjective functors''. In order to describe this phenomena in more detail, we recall that a category object $X_\bullet\in \on{Cat}(\mc{X})$ is called a groupoid object if all it's morphisms are invertible, i.e. $$X_2\to \lim_{\BDelta^{op}_{/\Lambda^2_0}} X_\bullet$$ is an equivalence,
 where $$\Lambda^2_0=\Delta^{0,1}\coprod_{\Delta^{0}}\Delta^{0,2}\subset \Delta^2.$$
 We let $\on{Gpd}(\mc{X})\subseteq\on{Cat}(\mc{X})$ denote the full subcategory spanned by the groupoid objects. The \emph{underlying groupoid functor} $\on{Gp}:\on{Cat}(\mc{X})\to \on{Gpd}(\mc{X})$ is any right adjoint to the inclusion. For a category object $X_\bullet$, one should think of $\on{Gp} X_\bullet$ as describing the ``maximal groupoid contained in $X_\bullet$'', which classifies the ``objects'' of the internal $(\infty,1)$-category $X_\bullet$.
 
The \emph{fully faithful and essentially surjective morphisms} (cf. \cite[Definition~1.2.12]{Lurie:2009uz}), are those morphisms of category objects $X_\bullet\to Y_\bullet$ in $\mc{X}$ which are 
	\begin{description}
	\item[fully faithful] the diagram
	$$\begin{tikzpicture}
	\mmat{m}{X_1&Y_1\\
	X_0\times X_0&Y_0\times Y_0\\};
	\draw[->] (m-1-1) --  (m-1-2);
	\draw[->] (m-1-1) --  (m-2-1);
	\draw[->] (m-1-2) --  (m-2-2);
	\draw[->] (m-2-1) -- (m-2-2);
\end{tikzpicture}$$
	is a pullback square, and
	\item[essentially surjective] the map $$\abs{\on{Gp}X_\bullet}\to\abs{\on{Gp}Y_\bullet}$$ 
	between the classifying spaces of objects is an equivalence, where $\abs{-}$ denotes the geometric realization:
	$$\abs{Z_\bullet}=\colim_{\BDelta^{op}}Z_\bullet$$ for any $Z_\bullet:\BDelta^{op}\to \mc{X}$.
	\end{description}
	Localizing along the fully faithful and essentially surjective morphisms of category objects, one obtains $\CSS(\mc{X})\subseteq \on{Cat}(\mc{X})$, the correct homotopy theory of $(\infty,1)$-category objects in $\mc{X}$. Following Rezk \cite{Rezk:2001ey} Lurie proves \cite[Theorem~1.2.13]{Lurie:2009uz} that $\CSS(\mc{X})\subseteq \on{Cat}(\mc{X})$ is equivalent to the full subcategory spanned by the \emph{complete Segal objects}: those category objects $X_\bullet\in \on{Cat}(\mc{X})$ such that $\on{Gp} X_\bullet$ is essentially constant (i.e. $\on{Gp} X_\bullet:\BDelta^{op}\to \mc{X}$ is equivalent to a constant functor).

 
To describe $(\infty,k)$-category objects in $\mc{X}$, will be interested in the following full subcategories of multisimplicial objects 
\begin{equation}\label{eq:CSShier}\CSS_k(\mc{X})\subseteq \on{Seg}_k(\mc{X})\subseteq \on{Cat}^k(\mc{X})\subseteq \on{Fun}\big((\BDelta^k)^{op},\mc{X}\big).\end{equation}

Here $\on{Cat}^k(\mc{X})$ is spanned by the $k$-uple category objects, i.e those multisimplicial objects $X_{\bullet,\dots,\bullet}$ such that for any $1\leq i\leq k$ and any $n_1,\dots,\hat n_i,\dots, n_k\geq 0$, the simplicial object 
\begin{subequations}\label{eq:kfoldSeg0}\begin{equation}\label{eq:kupleCat}X_{n_1,\dots, n_{i-1},\bullet,n_{i+1},\dots,n_k}:\BDelta^{op}\to\mc{X}\end{equation} is a category object. As before $X_{0,\dots,0}$ encodes the objects of the $k$-uple category internal to $\mc{X}$, but now each of $X_{1,0,\dots,0}$, $X_{0,1,0,\dots,0}$, $\dots,X_{0,\dots,0,1}$ encodes a different type of 1-morphism; while each of $X_{i_1,\dots,i_k}$  (with $0\leq i_1,\dots,i_k\leq 1$) represents a different type of $(i_1+\dots+i_k)$-morphism.
 As before, $\on{Cat}^k(\mc{X})$ does not model the correct homotopy theory of $k$-uple categories internal to $\mc{X}$; one must localize with respect to an appropriate class of ``fully faithful and essentially surjective functors''.

Next, $\on{Seg}_k(\mc{X})$ is spanned by the $k$-fold Segal objects (cf. \cite{Barwick:2005wl}), i.e. those $k$-uple category objects $X_{\bullet,\dots,\bullet}$ such that for every $1\leq i\leq k$, and any $n_1,\dots,n_{i-1}\geq 0$ the multisimplicial object
\begin{equation}\label{eq:kfoldSeg} X_{n_1,\dots,n_{i-1},0,\bullet,\dots,\bullet}\end{equation}
\end{subequations}
is equivalent to a constant functor. The idea behind this condition is that while a $k$-uple category object has ${k\choose i}$ different types of $i$-morphisms, there is only one non-trivial type of $i$-morphism in a $k$-fold Segal object. More specifically, $X_{(\bullet,\dots,\bullet)}:(\BDelta^{op})^n\to \mc{X}$ encodes the data of an $(\infty,k)$-category as follows: 
\begin{itemize}
	\item $X_{(0,\dots,0)}$ encodes the objects,
	\item $X_{(1,0,\dots,0)}$ encodes the $1$-morphisms, 
	\item $X_{(1,1,0,\dots,0)}$ encodes the $2$-morphisms,
	\item $\dots$
	\item and $X_{(1,\dots,1)}$ encodes the $k$-morphisms.
\end{itemize}
the remaining objects $X_{(n_1,\dots,n_k)}$ encode composable configurations of morphisms, while the homotopy coherent associative composition and unit are encoded in the various structural maps between the spaces $X_{(n_1,\dots,n_k)}$.
 
 Note that $\on{Seg}_k(\mc{X})$ does not model the correct homotopy theory of $k$-fold categories  internal to $\mc{X}$. However, when $\mc{X}$ is an $\infty$-topos (e.g. $\mc{X}=\Spc$), we may localize $\on{Seg}_k(\mc{X})$ with respect to an appropriate class of ``fully faithful and essentially surjective functors'', to obtain
$\CSS_k(\mc{X})$, which is spanned by those $k$-fold Segal objects which satisfy a certain \emph{completeness condition}; we refer the reader to  \cite{Barwick:2005wl,Barwick:2011wl,Lurie:2009uz,Haugseng:2014vw} for more details.


\begin{Lemma}\label{lem:ContOfCatObj}
Suppose $\mc{D}$ is an $\infty$-category,  $\mc{X}$ is a presentable $\infty$-category, and $\tilde{\mc{X}}$ is any reflective localization of $\on{Cat}^k(\mc{X})$ (the two main examples being $\tilde{\mc{X}}=\on{Seg}_k(\mc{X})$, or when $\mc{X}$ is an $\infty$-topos, $\tilde{\mc{X}} = \CSS_k(\mc{X})$).
Then a functor 
\begin{equation}\label{eq:CSSFunct}\mc{D}\xrightarrow{d\mapsto F(d)_{(\bullet,\dots,\bullet)}}\tilde{\mc{X}}\end{equation}
is continuous if and only if each of the composite functors 
\begin{equation}\label{eq:CSSMorphFunct}\mc{D}\xrightarrow{d\mapsto F(d)_{(i_1,\dots,i_k)}} \mc{X},\quad 0\leq i_1,\dots, i_k\leq 1\end{equation}
obtained by evaluating at $(i_1,\dots,i_k)\in \BDelta^k$ for $0\leq i_1,\dots, i_k\leq 1$, are continuous.
\end{Lemma}
\begin{proof}
Without loss of generality, we may take $\tilde{\mc{X}}=\on{Cat}^k(\mc{X})$. Recall that
	$\on{Cat}^k(\mc{X})$ is a reflective localization of $\on{Fun}\big((\BDelta^k)^{op},\mc{X}\big)$ (cf. \cite{Lurie:2009uz}), so \eqref{eq:CSSFunct} is continuous if and only if the composite functor
	$$\mc{D}\xrightarrow{d\mapsto F(d)_{(\bullet,\dots,\bullet)}}\on{Cat}^k(\mc{X})\hookrightarrow \on{Fun}\big((\BDelta^k)^{op},\mc{X}\big)$$
	is continuous. Since limits in functor $\infty$-categories are detected pointwise (cf. \cite[Corollary~5.1.2.3]{Lurie:2009un}), it follows that \eqref{eq:CSSFunct} is continuous if and only if the composite functors
	\begin{equation}\label{eq:EvalAtNs}\mc{D}\xrightarrow{d\mapsto F(d)_{(n_1,\dots,n_k)}} \mc{X}\end{equation}
	obtained by evaluating at any  $(n_1,\dots,n_k)\in \BDelta^k$
are continuous. This proves the only if part of the statement.

Now we prove the if part of the statement. Let $i:\mbb{Morph}^k\hookrightarrow \BDelta^k$ denote the inclusion of the full subcategory spanned by $(i_1,\dots,i_k)\in \BDelta^k$, where $0\leq i_1,\dots,i_k\leq 1$. Then by assumption, \eqref{eq:CSSMorphFunct} and hence the restricted functor 
$$\mc{D}\xrightarrow{d\mapsto F(d)_{(\bullet,\dots,\bullet)}}\on{Cat}^k(\mc{X})\hookrightarrow \on{Fun}\big((\BDelta^k)^{op},\mc{X}\big)\xrightarrow{i^*} \on{Fun}\big((\mbb{Morph}^k)^{op},\mc{X}\big)$$
is continuous.

A map of simplices $(\phi:[n]\to [m])\in\BDelta$ is said to be \emph{inert} if it is the inclusion of a full sub-interval, i.e. $\phi(i+1)=\phi(i)+1$ for every $i\in [n]$ (cf. \cite{Haugseng:2014vw,Barwick:2013vc}). We let $j:\BDelta_{\on{int}}\hookrightarrow\BDelta$ denote the inclusion of the wide subcategory containing only the inert maps.
For any $(n_1,\dots,n_k)\in \BDelta^k$, let $$\mbb{Spine}(n_1,\dots,n_k)=\mbb{Morph}^k\times_{\BDelta_{\on{int}}^k}\big(\BDelta_{\on{int}}^k\big)_{/(n_1,\dots,n_k)}.$$
Then the Segal conditions imply that for any  $(n_1,\dots,n_k)\in \BDelta^k$, and $d\in \mc{D}$, the object
$F(d)_{(n_1,\dots,n_k)}\in \mc{X}$ is a limit for the composite functor 
$$\mbb{Spine}(n_1,\dots,n_k)^{op}\to (\BDelta^k)^{op}\xrightarrow{F(d)}\mc{X}$$
(cf. \cite[Lemma~2.27]{Haugseng:2014vw}). 
Equivalently, the restriction $F(d)\rvert_{\big(\BDelta_{\on{int}}^k)^{op}}$ is a right Kan extension along $\delta:\mbb{Spine}(n_1,\dots,n_k)^{op}\to (\BDelta_{\on{int}}^k)^{op}$.

Let  $$\on{Fun}\big((\BDelta_{\on{int}}^k)^{op},\mc{X}\big)\xrightarrow{\delta_*} \on{Fun}\big((\mbb{Morph}^k)^{op},\mc{X}\big)$$ denote the right adjoint (the global right Kan extension) to the pullback $\delta^*$, then the composite
$$\mc{D}\xrightarrow{d\mapsto F(d)_{(\bullet,\dots,\bullet)}}\on{Cat}^k(\mc{X})\hookrightarrow \on{Fun}\big((\BDelta^k)^{op},\mc{X}\big)\xrightarrow{i^*} \on{Fun}\big((\mbb{Morph}^k)^{op},\mc{X}\big)\xrightarrow{\delta_*}\on{Fun}\big((\BDelta_{\on{int}}^k)^{op},\mc{X}\big)$$
is continuous. However, by assumption, this functor is equivalent to the restricted functor 
$$\mc{D}\xrightarrow{d\mapsto F(d)_{(\bullet,\dots,\bullet)}}\on{Cat}^k(\mc{X})\hookrightarrow \on{Fun}\big((\BDelta^k)^{op},\mc{X}\big)\xrightarrow{j^*} \on{Fun}\big((\BDelta_{\on{int}}^k)^{op},\mc{X}\big).$$
It follows that each \eqref{eq:EvalAtNs} is continuous, whence \eqref{eq:CSSFunct} is continuous.
	\end{proof}

\subsection{The Sheaf of Complete $k$-Fold Segal Objects}

Given two morphisms of simplicial sets $X\to S$ and $Y\to S$, we  let $Y^X\to S$ denote the simplicial set satisfying the universal property that for any morphism of simplicial sets $K\to S$, commutative diagrams of the form 
$$\begin{tikzpicture}\node (y) at (-2,1) {$K\times_S X$};
\node (q) at (2,1) {$Y$};
\node (css) at (0,0) {$S$};
\draw[->] (y) -- (q);
\draw[->] (y) -- (css);
\draw[->] (q) -- (css);
\end{tikzpicture}$$
correspond to diagrams of the form
 $$\begin{tikzpicture}\node (y) at (-2,1) {$K$};
\node (q) at (2,1) {$Y^X$};
\node (css) at (0,0) {$S$};
\draw[->] (y) -- (q);
\draw[->] (y) -- (css);
\draw[->] (q) -- (css);
\end{tikzpicture}$$
In particular, when $Y\to S$ is a coCartesian fibration, and $X\to S$ is a Cartesian fibration, then $Y^X\to S$ is a coCartesian fibration satisfying
$$\on{Fun}_{S}(K,Y^X)\cong\on{Fun}_S(K\times_S X,Y),$$
(see \cite[Corollary~3.2.2.13]{Lurie:2009un} for more details).

Let $\widehat{\Cat_\infty}$ denote the $\infty$-category of (not necessarily small) $\infty$-categories, and  $\imath:\LTop\hookrightarrow\widehat{\Cat_\infty}$ denote the subcategory consisting of $\infty$-topoi and geometric morphisms (functors which preserve small colimits and finite limits). Notice that $\imath$ factors through the subcategory of presentable $\infty$-categories and left adjoints.
Let $\imath^*\mc{Z}\to \LTop$ denote the (cannonical) presentable fibration classified by $\imath$ (cf. \cite[Proposition~5.5.3.3]{Lurie:2009un}).\footnote{Recall that a fibration is presentable if it is both a Cartesian and a coCartesian fibration each of whose fibres are presentable $\infty$-categories.} We define a presentable fibration
$$
k\textrm{-Simpl}(\LTop):=\imath^*\mc{Z}^{(\LTop\times (\BDelta^k)^{op})}\xrightarrow{p} \LTop, $$
whose fibre over any $\infty$-topos $\mc{X}$ is equivalent to
 $\on{Fun}\big((\BDelta^k)^{op},\mc{X}\big)$
and which associates to any geometric morphism of $\infty$-topoi $f^*:\mc{X}\leftrightarrows\mc{Y}:f_*$ the adjunction given by composition with $f^*$ (resp. $f_*$)
$$(f^*)_!:\on{Fun}\big((\BDelta^k)^{op},\mc{X}\big)\leftrightarrows \on{Fun}\big((\BDelta^k)^{op},\mc{Y}\big):(f_*)_!.$$


Suppose that $X_{\bullet,\dots,\bullet}\in k\textrm{-Simpl}(\LTop)$ is a vertex lying over $\mc{X}=p(X_{\bullet,\dots,\bullet})$. We say that $X_{\bullet,\dots,\bullet}$ is a \emph{complete Segal object} if it lies in the essential image of $\CSS_k(\mc{X})\hookrightarrow p^{-1}(\mc{X})$. We define ${\int\CSS_k}$ to be the full subcategory of $k\textrm{-Simpl}(\LTop)$ spanned by the complete Segal objects.

\begin{Lemma}\label{lem:CSSPres}
	${\int\CSS_k}\to \LTop$ is a presentable fibration.
\end{Lemma}
\begin{proof}
	We begin by showing that ${\int\CSS_k}$ is a Cartesian fibration. It suffices to show that for any complete Segal object $Y_{\bullet,\dots,\bullet}\in k\textrm{-Simpl}(\LTop)$ and any $p$-Cartesian morphism $\tilde f:X_{\bullet,\dots,\bullet}\to Y_{\bullet,\dots,\bullet}$, the vertex $X_{\bullet,\dots,\bullet}\in k\textrm{-Simpl}(\LTop)$ is also a complete Segal object. Let $f^*:\mc{X}\to\mc{Y}$ denote the image of $\tilde f$ under $p$, and let $\mc{X}\leftarrow \mc{Y}:f_*$ denote a right adjoint to $f^*$. Then $X_{\bullet,\dots,\bullet}\cong (f_*)_!(Y_{\bullet,\dots,\bullet})$ by the construction of $p$, and the latter is a complete Segal object by \cite[Proposition~2.20]{Haugseng:2014vw}.
	
	Indeed, this shows that ${\int\CSS_k}\to\LTop$ is a Cartesian fibration classifed by a functor 
	$$\chi:\LTop^{op}\to \widehat{\Cat_\infty}$$
	such that 
	\begin{itemize}
		\item for every $\infty$-topos $\mc{X}\in\LTop$, the image $\chi(\mc{X})$ is equivalent to the presentable $\infty$-category $\CSS_k(\mc{X})$, and
		\item  for every geometric morphism $f^*:\mc{X}\leftrightarrows\mc{Y}:f_*$, the functor $$\CSS_k(\mc{X})\leftarrow \CSS_k(\mc{Y}):\chi(f^*)$$ is equivalent to 
	$$\CSS_k(\mc{X})\leftarrow \CSS_k(\mc{Y}):(f_*)_!,$$
	which has a left adjoint (cf. \cite[Proposition~2.20]{Haugseng:2014vw}).
	\end{itemize}

	It follows that ${\int\CSS_k}\to \LTop$ is a presentable fibration (cf. \cite[Proposition~5.5.3.3]{Lurie:2009un}).
	
\end{proof}

\begin{remark}
By construction, the objects $X_{\bullet,\dots,\bullet}$ in ${\int\CSS_k}$ over an $\infty$-topos $\mc{X}$ can be identified with complete Segal objects in $\mc{X}$, and 	morphisms $X_{\bullet,\dots,\bullet}\to Y_{\bullet,\dots,\bullet}$ in ${\int\CSS_k}$ over a geometric morphism $f^*:\mc{X}\leftrightarrows\mc{Y}:f_*$ of $\infty$-topoi  can be identified with either
\begin{itemize}
	\item[(a)] morphisms $X_{\bullet,\dots,\bullet}\to (f_*)_!(Y_{\bullet,\dots,\bullet})$ in $\CSS_k(\mc{X})$, or
	\item[(b)] morphisms $L_{k,\mc{Y}}(f^*)_!(X_{\bullet,\dots,\bullet})\to Y_{\bullet,\dots,\bullet}$ in $\CSS_k(\mc{Y})$,
\end{itemize}
where $L_{k,\mc{Y}}:\on{Seg}_k(\mc{Y})\to \CSS_k(\mc{Y})$ is the localization functor which sends a $k$-fold Segal object in $\mc{Y}$ to its completion. The equivalence between morphisms of types (a) and (b) is given by the adjunction
$$L_{k,\mc{Y}}(f^*)_!:\CSS_k(\mc{X})\leftrightarrows\CSS_k(\mc{Y}):(f_*)_!$$
of \cite[Proposition~2.20]{Haugseng:2014vw}.
\end{remark}

Recall that a geometric morphism $f^*:\mc{X}\rightleftarrows \mc{Y}:f_*$ is said to be \'etale if it admits a factorization 
$$f^*:\mc{X}\underset{f'_*}{\overset{f'^*}{\rightleftarrows}}\mc{X}_{/U}\cong \mc{Y}:f_*$$
for some object $U\in \mc{X}$. We let $\LTop_{\acute{e}t}\subset\LTop$ denote the subcategory spanned by the \'etale geometric morphisms, and we define $\int^{\acute{e}t}\CSS_k\to \LTop_{\acute{e}t}$ to be the presentable fibration fitting into the pullback square:
$$\begin{tikzpicture}
	\mmat{m}{\int^{\acute{e}t}\CSS_k&\int\CSS_k\\
	\LTop_{\acute{e}t}&\LTop\\};
	\draw[->] (m-1-1) --  (m-1-2);
	\draw[->] (m-1-1) --  (m-2-1);
	\draw[->] (m-1-2) --  (m-2-2);
	\draw[->] (m-2-1) -- (m-2-2);
\end{tikzpicture}$$

\begin{remark}
Since $\int\CSS_k\to \LTop$ and $\int^{\acute{e}t}\CSS_k\to \LTop_{\acute{e}t}$ are both presentable fibrations\cite[Corollary~4.3.1.11]{Lurie:2009un}	implies that they both admit all small relative limits and colimits. Since $\LTop$ admits all small limits and colimits (cf. \cite[\S~6.3]{Lurie:2009un}), it follows that $\int\CSS_k$ admits all small limits and colimits, and that the functor $\int\CSS_k\to \LTop$ preserves those limits and colimits (cf. \cite[Lemma~9.8]{Gepner:2015ww}).

Finally, \cite[Theorem~6.3.5.13]{Lurie:2009un} implies that $\LTop_{\acute{e}t}\subset\LTop$ is closed under small limits, which implies that $\int^{\acute{e}t}\CSS_k\subset\int\CSS_k$ is also closed under small limits.
\end{remark}

As explained in \cite[Remark~6.3.5.10]{Lurie:2009un} for any $\infty$-topos $\mc{X}$, the Cartesian fibration $$\on{Fun}(\Delta^1,\mc{X})\to \on{Fun}(\{1\},\mc{X})\cong\mc{X}$$ is classified by a functor 
\begin{subequations}\label{eq:FXtopostoEtale}\begin{equation}\label{eq:EtaleOverX}\mc{X}^{op}\xrightarrow{\begin{array}{rcl}U&\mapsto& \mc{X}_{/U}\\(f:U\to V)&\mapsto &(f^*:\mc{X}_{/V}\leftrightarrows\mc{X}_{/U}:f_*)\end{array}} \LTop_{\acute{e}t},\end{equation}
 which factors as 
\begin{equation}\label{eq:EtaleOverXEquiv}\mc{X}^{op}\xrightarrow{\cong} (\LTop_{\acute{e}t})_{\mc{X}/}\to \LTop_{\acute{e}t},\end{equation}\end{subequations}
where the first functor is an equivalence of categories.

\begin{Definition}[{\cite[Notation~6.3.5.19]{Lurie:2009un}}]
	Given a functor $F:\LTop\to \mc{C}$, let $F_{\mc{X}}:\mc{X}^{op}\to \mc{C}$ denote the composite
	$$\mc{X}^{op}\to \LTop_{\acute{e}t}\subseteq \LTop\xrightarrow{F}\mc{C}.$$
	We say that $F$ is a \emph{sheaf} if for every $\infty$-topos $\mc{X}$, the composite functor $F_{\mc{X}}$ preserves small limits.
\end{Definition}

\begin{Theorem}\label{thm:CSSaSheaf}
The functor \begin{subequations}\begin{equation}\label{eq:CSSetFun}\CSS_k:\LTop_{\acute{e}t}\xrightarrow{\mc{X}\mapsto \CSS_k(\mc{X})}\widehat{\Cat_\infty}\end{equation}
 classifying $\int^{\acute{e}t}\CSS_k\to \LTop_{\acute{e}t}$ preserves small limits. In particular, \begin{equation}\label{eq:CSSFun}\CSS_k:\LTop\xrightarrow{\mc{X}\mapsto \CSS_k(\mc{X})}\widehat{\Cat_\infty}\end{equation}\end{subequations}
 is a sheaf.
\end{Theorem}
\begin{proof}
By definition, for any \'etale geometric morphism  $f^*:\mc{X}\rightleftarrows \mc{X}_{/U}:f_*$, the canonical projection $f_!:\mc{X}_{/U}\to\mc{X}$ forms part of an adjoint triple, 
$$(f_!\dashv f^*\dashv f_*):\mc{X}_{/U}\underset{\xrightarrow{f_*}}{\overset{\xrightarrow{f_!}}{\xleftarrow{f^*}}}\mc{X}.$$
Moreover, the forgetful functor from the over category $f_!:\mc{X}_{/U}\to \mc{X}$ preserves pullbacks, so that $f_!:\mc{X}_{/U}\leftrightarrows\mc{X}:f^*$ is a \emph{pseudo-geometric morphism} (cf. \cite{Haugseng:2014vw}). In particular, $(f^*)_!:\on{Fun}\big((\BDelta^k)^{op},\mc{X})\to \on{Fun}\big((\BDelta^k)^{op},\mc{X}_{/U})$ preserves complete $k$-fold Segal objects (cf. \cite[Proposition~2.20]{Haugseng:2014vw}); and consequently the inclusion 
$$\begin{tikzpicture}
	\node (m-1-1) at (-1.5,1.5) {$\int^{\acute{e}t}\CSS_k$};
	\node (m-1-2) at (1.5,1.5) {$k\textrm{-Simpl}(\LTop_{\acute{e}t})$};
	\node (m-2-2) at (0,0) {$\LTop_{\acute{e}t}$};
	\draw[->] (m-1-1) --  (m-1-2);
	\draw[->] (m-1-1) --  (m-2-2);
	\draw[->] (m-1-2) --  (m-2-2);
\end{tikzpicture}$$
preserves both Cartesian and coCartesian edges. Therefore, 
\begin{itemize}
\item 	$\CSS_k:\LTop_{\acute{e}t}\to \widehat{\Cat_\infty}$ is a (fully faithful) subfunctor of the composite functor
$$k\textrm{-Simpl}':\LTop_{\acute{e}t}\hookrightarrow\widehat{\Cat_\infty}\xrightarrow{\mc{C}\mapsto \on{Fun}((\BDelta^k)^{op},\mc{C})}\widehat{\Cat_\infty},$$ classifying $k\textrm{-Simpl}(\LTop_{\acute{e}t})\to \LTop_{\acute{e}t}$, and
\item this latter functor is continuous (cf. \cite[Proposition~6.3.2.3, Theorem~6.3.5.13]{Lurie:2009un}).
\end{itemize}
We will leverage these facts to show that $\CSS_k\rvert_{\LTop_{\acute{e}t}}$ is continuous. For simplicity of exposition, we restrict to the case that $k=1$.

Suppose that $q:I\to \LTop_{\acute{e}t}$ is a diagram. Then we may identify the limit of $1\textrm{-Simpl}'\circ q$ with the $\infty$-category 
$$\lim \big(1\textrm{-Simpl}'\circ q\big)\subset \on{Fun}_I\bigg(I, q^*\big(1\textrm{-Simpl}(\LTop_{\acute{e}t})\big)\bigg)$$
 of coCartesian sections of the pulled-back presentable fibration $q^*\big(1\textrm{-Simpl}(\LTop_{\acute{e}t})\big)\to I$. Similarly, we may identify the  limit of $\CSS_1\circ q$ with the $\infty$-category 
$$\lim \big(\CSS_1\circ q\big)\subset \on{Fun}_I\bigg(I, q^*\big(\int^{\acute{e}t}\CSS_1\big)\bigg)$$
of coCartesian sections of the pulled-back presentable fibration $q^*\big(\int^{\acute{e}t}\CSS_1)\big)\to I$. 

Now let $\mc{X}\cong\lim q\in \LTop_{\acute{e}t}$ be the limit of $q$. Then $\CSS_1(\mc{X})$ is the accessible localization of $\on{Fun}\big((\BDelta)^{op},\mc{X})\cong\lim \big(1\textrm{-Simpl}'\circ q\big)$ spanned by those objects which satisfy
\begin{enumerate}
	\item the Segal conditions \eqref{eq:Segal} which specify the category objects,\footnote{when $k> 1$, one also has constancy conditions \eqref{eq:kfoldSeg0}, which are likewise given as limits.} and 
	\item the completeness conditions; namely (in the case that $k=1$) that $\on{Gp} X_\bullet:\BDelta^{op}\to \mc{X}$ is equivalent to the constant functor. 
\end{enumerate} 
	So we have full and faithful inclusions of both $\CSS_1(\mc{X})$ and $\lim(\CSS_1\circ q)$ into the $\infty$-category, $\on{Fun}\big((\BDelta)^{op},\mc{X})$, of co-Cartesian sections of $q^*\big(1\textrm{-Simpl}(\LTop_{\acute{e}t})\big)\to I$. Using the universal property for the limit yields a diagram of full and faithful inclusions:
	$$\CSS_1(\mc{X})\hookrightarrow\lim(\CSS_1\circ q)\hookrightarrow\on{Fun}\big((\BDelta)^{op},\mc{X}).$$
Thus, it suffices to show that any coCartesian section of $q^*\big(\int^{\acute{e}t}\CSS_1)\big)\to I$ lies in the essential image of the leftmost functor - i.e. satisfies conditions (1) and (2).
 As a first step, notice that 
 $q^*\big(\int^{\acute{e}t}\CSS_1)\big)\to I$ 
 satisfies conditions (1) and (2) fibrewise.

For every $i\in I$, let $\pi_i^*:\mc{X}\leftrightarrows q(i):{\pi_i}_*$ denote the \'etale geometric morphism fitting into the limit cone. Recall that left adjoints of \'etale geometric morphisms $f^*:\mc{Y}\to \mc{Z}$ are continuous. Now, since the conditions for a simplicial object $X_{\bullet}\in \on{Fun}\big((\BDelta)^{op},\mc{X})$ to be a category object are given in terms of limits, Theorem~\ref{thm:LimInLim} implies that $X_{\bullet}$ is a category object if and only of each of the simplicial objects $(\pi_i^*)_!\big(X_{\bullet}\big)$ are category objects. 

Next, \cite[Proposition~2.20]{Haugseng:2014vw} implies that left adjoints of \'etale geometric morphisms $f^*:\mc{Y}\to \mc{Z}$ commute with the \emph{underlying groupoid} functors, i.e. 
$$\begin{tikzpicture}
	\mmat{m}{\on{Seg}(\mc{Y})&\on{Seg}(\mc{Z})\\
	\on{Gpd}(\mc{Y})&\on{Gpd}(\mc{Z})\\};
	\draw[->] (m-1-1) -- node {$(f^*)_!$} (m-1-2);
	\draw[->] (m-1-1) -- node[swap] {$\on{Gp}$}  (m-2-1);
	\draw[->] (m-1-2) -- node {$\on{Gp}$} (m-2-2);
	\draw[->] (m-2-1) -- node[swap] {$(f^*)_!$} (m-2-2);
\end{tikzpicture}$$
commutes. Suppose now that $X_{\bullet}\in\on{Seg}(\mc{X})$ is a category object, which we may identify with a coCartesian section $X_\bullet':I\to q^*\big(1\textrm{-Simpl}(\LTop_{\acute{e}t})\big)$.
Applying Theorem~\ref{thm:limadj}, we see that the underlying groupoid $\on{Gp}X_{\bullet}$ can be identified with the coCartesian section $\on{Gp}\circ X_\bullet':I\to q^*\big(1\textrm{-Simpl}(\LTop_{\acute{e}t})\big)$
obtained by applying the underlying groupoid functor fibrewise.\footnote{See \cite[Proposition~7.3.2.6]{Lurie:0uBkkKfz} to confirm that this right adjoint can be applied fibrewise in a coherent manner.} Consequently, for every $i\in I$ we have $$(\pi_i^*)_!\on{Gp}X_{\bullet}\cong \on{Gp}\big((\pi_i^*)_!X_{\bullet}\big).$$ Therefore $$\on{Gp}X_{\bullet}:\BDelta^{op}\to \mc{X}$$ is essentially constant if and only if each $$\on{Gp}\big((\pi_i^*)_!X_{\bullet}\big):\BDelta^{op}\to q(i)$$ is essentially constant.



Thus, we have shown $\CSS_1(\mc{X})\hookrightarrow\lim \big(\CSS_1\circ q\big)$ is an equivalence, which proves that \eqref{eq:CSSetFun} preserves small limits. 

Now for any $\infty$-topos $\mc{X}$, the functor $\mc{X}^{op}\to \LTop_{\acute{e}t}$ given by \eqref{eq:EtaleOverX} factors as an equivalence followed by the forgetful functor from an undercategory \eqref{eq:EtaleOverXEquiv}; hence it preserves small limits (cf. \cite[Proposition~1.2.13.8]{Lurie:2009un}). It follows that the composite $$\mc{X}^{op}\to \LTop_{\acute{e}t}\xrightarrow{\CSS_k}\widehat{\Cat_\infty}$$ also preserves small limits, so \eqref{eq:CSSFun} is a sheaf.
\end{proof}

 \section{$\infty$-categories of spans.}
Let $\mc{C}$ be an $\infty$-category with pullbacks. In \cite{Barwick:2013th}, Barwick introduces the $\infty$-category $\on{Span}(\mc{C})$, which has the same space of objects as $\mc{C}$, but whose morphisms between two objects $c_0,c_1\in \mc{C}$ is the space of diagrams in $\mc{C}$ of the form
$$\begin{tikzpicture}
 \node  (m-1-1)  at (0,1) {$x$};
 \node  (m-1-2)  at (1,0) {$c_0$};
 \node  (m-2-1)  at (-1,0) {$c_1$};
\draw[->] (m-1-1) -- (m-1-2);
\draw[->] (m-1-1) -- (m-2-1);
\end{tikzpicture}$$
That is, spans $c_0\nrightarrow c_1$ in $\mc{C}$.
Composition of two such morphisms is given by taking the fibred product. 
Haugseng \cite{Haugseng:2014vw} extends this construction, introducing an $(\infty,k)$-category $\on{Span}_k(\mc{C})$ of iterated spans in $\mc{C}$, whose 2-morphisms are spans between spans, and so forth. In this section, we show that the functor $\mc{C}\to \on{Span}_k(\mc{C})$ depends continuously on $\mc{C}$. 
\subsection{Continuity of the formation of $\infty$-categories of iterated spans.}

We now briefly recall Haugseng's construction. Let $\BSigma^n$ denote the partially ordered set whose objects are pairs of numbers $(i,j)$ such that $0\leq i\leq j\leq n$, and $(i,j)\leq (i',j')$ if $i\leq i'$ and $j'\leq j$. We may picture the poset $\BSigma^n$ (using Barwick's notation $\bar p=n-p$) as follows:
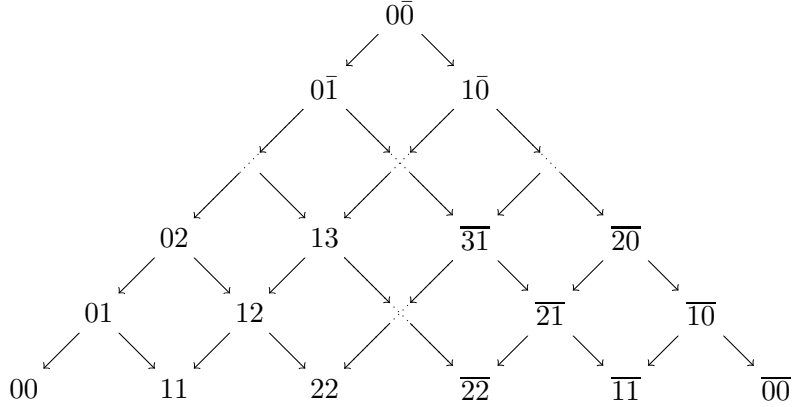
\begin{center}
\begin{tikzpicture}
	\path [use as bounding box] (-5.5,0.5) rectangle (5.5,6.5);
	\begin{pgfonlayer}{nodelayer}
		\node [style=hidden] (0) at (0, 6) {$0\bar 0$};
		\node [style=hidden] (1) at (-1, 5) {$0\bar1$};
		\node [style=hidden] (2) at (1, 5) {$1\bar 0$};
		\node [style=hidden] (3) at (0, 4) {};
		\node [style=hidden] (4) at (-2, 4) {};
		\node [style=hidden] (5) at (2, 4) {};
		\node [style=hidden] (6) at (1, 3) {$\overline{31}$};
		\node [style=hidden] (7) at (-1, 3) {$13$};
		\node [style=hidden] (8) at (-3, 3) {$02$};
		\node [style=hidden] (9) at (3, 3) {$\overline{20}$};
		\node [style=hidden] (10) at (-4, 2) {$01$};
		\node [style=hidden] (11) at (-5, 1) {$00$};
		\node [style=hidden] (12) at (-3, 1) {$11$};
		\node [style=hidden] (13) at (-2, 2) {$12$};
		\node [style=hidden] (14) at (-1, 1) {$22$};
		\node [style=hidden] (15) at (0, 2) {};
		\node [style=hidden] (16) at (1, 1) {$\overline{22}$};
		\node [style=hidden] (17) at (2, 2) {$\overline{21}$};
		\node [style=hidden] (18) at (3, 1) {$\overline{11}$};
		\node [style=hidden] (19) at (4, 2) {$\overline{10}$};
		\node [style=hidden] (20) at (5, 1) {$\overline{00}$};
	\end{pgfonlayer}
	\begin{pgfonlayer}{edgelayer}
		\draw [style=arrow] (10) to (12);
		\draw [style=arrow] (10) to (11);
		\draw [style=arrow] (8) to (10);
		\draw [style=arrow] (8) to (13);
		\draw [style=arrow] (13) to (12);
		\draw [style=arrow] (13) to (14);
		\draw [style=arrow] (4) to (8);
		\draw [style=arrow] (1) to (4);
		\draw [style=arrow] (0) to (1);
		\draw [style=arrow] (0) to (2);
		\draw [style=arrow] (1) to (3);
		\draw [style=arrow] (2) to (3);
		\draw [style=arrow] (3) to (7);
		\draw [style=arrow] (7) to (13);
		\draw [style=arrow] (7) to (15);
		\draw [style=arrow] (15) to (14);
		\draw [style=arrow] (6) to (15);
		\draw [style=arrow] (3) to (6);
		\draw [style=arrow] (2) to (5);
		\draw [style=arrow] (5) to (6);
		\draw [style=arrow] (5) to (9);
		\draw [style=arrow] (9) to (17);
		\draw [style=arrow] (6) to (17);
		\draw [style=arrow] (15) to (16);
		\draw [style=arrow] (17) to (16);
		\draw [style=arrow] (17) to (18);
		\draw [style=arrow] (9) to (19);
		\draw [style=arrow] (19) to (18);
		\draw [style=arrow] (19) to (20);
		\draw [style=arrow] (4) to (7);
	\end{pgfonlayer}
\draw[dotted,-] (1) -- (6);
\draw[dotted,-] (2) -- (7);
\draw[dotted,-] (7) -- (16);
\draw[dotted,-] (6) -- (14);
\draw[dotted,-] (1) -- (8);
\draw[dotted,-] (2) -- (9);
\end{tikzpicture}
\end{center}
For any map of totally ordered sets $\phi:[n]\to [m]$, the map $(i,j)\to \big(\phi(i),\phi(j)\big)$ induces a monotone map $\BSigma^n\to\BSigma^m$; and thus we have a functor $\BSigma^\bullet:\BDelta\to\Cat_\infty$. Similarly, taking $k$-fold product, $\BSigma^{n_1,\dots,n_k}:=\BSigma^{n_1}\times\cdots\times\BSigma^{n_k}$ defines a functor \begin{equation}\label{eq:BSigFunct}\BSigma^{\bullet,\dots,\bullet}:\BDelta^k\to\Cat_\infty.\end{equation}

Suppose that $\mc{C}$ is an $\infty$-category with finite limits. We will be interested in functors $f:\BSigma^{n_1,\dots,n_k}\to \mc{C}$. 
We let $\BLambda^k\subseteq \BSigma^k$ denote the full subcategory
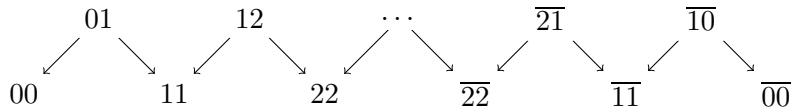
\begin{center}
\begin{tikzpicture}
	\path [use as bounding box] (-5.5,0.5) rectangle (5.5,2.5);
	\begin{pgfonlayer}{nodelayer}
		\node [style=hidden] (0) at (-4, 2) {${01}$};
		\node [style=hidden] (1) at (-5, 1) {${00}$};
		\node [style=hidden] (2) at (-3, 1) {${11}$};
		\node [style=hidden] (3) at (-2, 2) {${12}$};
		\node [style=hidden] (4) at (-1, 1) {${22}$};
		\node [style=hidden] (5) at (0, 2) {$\dots$};
		\node [style=hidden] (6) at (1, 1) {${\overline{22}}$};
		\node [style=hidden] (7) at (2, 2) {${\overline{21}}$};
		\node [style=hidden] (8) at (3, 1) {${\overline{11}}$};
		\node [style=hidden] (9) at (4, 2) {${\overline{10}}$};
		\node [style=hidden] (10) at (5, 1) {${\overline{00}}$};
	\end{pgfonlayer}
	\begin{pgfonlayer}{edgelayer}
		\draw [style=arrow] (0) to (2);
		\draw [style=arrow] (0) to (1);
		\draw [style=arrow] (3) to (2);
		\draw [style=arrow] (3) to (4);
		\draw [style=arrow] (5) to (4);
		\draw [style=arrow] (5) to (6);
		\draw [style=arrow] (7) to (6);
		\draw [style=arrow] (7) to (8);
		\draw [style=arrow] (9) to (8);
		\draw [style=arrow] (9) to (10);
	\end{pgfonlayer}
\end{tikzpicture}
\end{center}
 spanned by those pairs $(i,j)$ with $j-i\leq1$. Similarly, we define $\BLambda^{n_1,\dots,n_k}:=\BLambda^{n_1}\times\cdots\times\BLambda^{n_k}$, and let $\iota_{n_1,\dots,n_k}:\BLambda^{n_1,\dots,n_k}\to \BSigma^{n_1,\dots,n_k}$ denote the inclusion. 
 \begin{Definition}[\cite{Barwick:2013th,Haugseng:2014vw}]\label{def:CartFunct}
We say that a functor $f:\BSigma^{n_1,\dots,n_k}\to \mc{C}$ is \emph{Cartesian} if it is a right Kan extension of $f\circ \iota_{n_1,\dots,n_k}$, and we let $\on{Fun}^{\BSigma\textrm{-Cart}}\big(\BSigma^{n_1,\dots,n_k},\mc{C}\big)\subseteq  \on{Fun}\big(\BSigma^{n_1,\dots,n_k},\mc{C}\big)$ denote the full subcategory spanned by the Cartesian functors. We let $$\on{Map}^{\BSigma\textrm{-Cart}}\big(\BSigma^{n_1,\dots,n_k},\mc{C}\big):=\iota\on{Fun}^{\BSigma\textrm{-Cart}}\big(\BSigma^{n_1,\dots,n_k},\mc{C}\big)\subset \on{Fun}^{Cart}\big(\BSigma^{n_1,\dots,n_k},\mc{C}\big)$$ denote the classifying space of Cartesian functors.\footnote{i.e. the largest Kan complex in $\on{Fun}^{\BSigma\textrm{-Cart}}\big(\BSigma^{n_1,\dots,n_k},\mc{C}\big)$.}
 \end{Definition}

For example, when $k=1$, a Cartesian functor $f:\BSigma^{n_1,\dots,n_k}\to \mc{C}$ is a diagram of the form
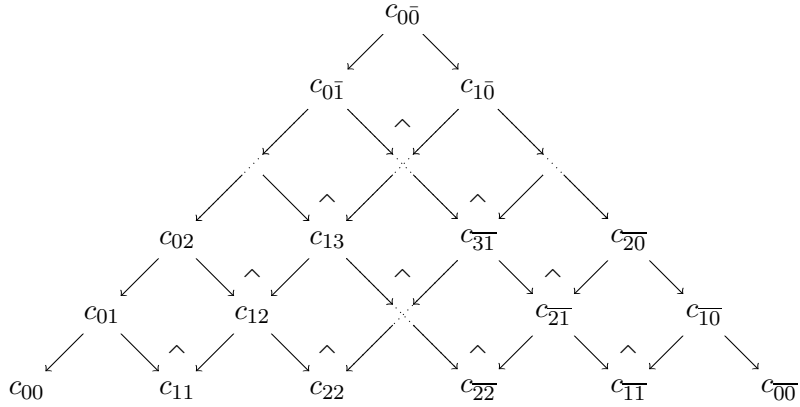
\begin{center}
\begin{tikzpicture}
	\path [use as bounding box] (-5.5,0.5) rectangle (5.5,6.5);
	\begin{pgfonlayer}{nodelayer}
		\node [style=hidden] (0) at (0, 6) {$c_{0\bar 0}$};
		\node [style=hidden] (1) at (-1, 5) {$c_{0\bar1}$};
		\node [style=hidden] (2) at (1, 5) {$c_{1\bar 0}$};
		\node [style=hidden] (3) at (0, 4) {};
		\node [style=hidden] (4) at (-2, 4) {};
		\node [style=hidden] (5) at (2, 4) {};
		\node [style=hidden] (6) at (1, 3) {$c_{\overline{31}}$};
		\node [style=hidden] (7) at (-1, 3) {$c_{13}$};
		\node [style=hidden] (8) at (-3, 3) {$c_{02}$};
		\node [style=hidden] (9) at (3, 3) {$c_{\overline{20}}$};
		\node [style=hidden] (10) at (-4, 2) {$c_{01}$};
		\node [style=hidden] (11) at (-5, 1) {$c_{00}$};
		\node [style=hidden] (12) at (-3, 1) {$c_{11}$};
		\node [style=hidden] (13) at (-2, 2) {$c_{12}$};
		\node [style=hidden] (14) at (-1, 1) {$c_{22}$};
		\node [style=hidden] (15) at (0, 2) {};
		\node [style=hidden] (16) at (1, 1) {$c_{\overline{22}}$};
		\node [style=hidden] (17) at (2, 2) {$c_{\overline{21}}$};
		\node [style=hidden] (18) at (3, 1) {$c_{\overline{11}}$};
		\node [style=hidden] (19) at (4, 2) {$c_{\overline{10}}$};
		\node [style=hidden] (20) at (5, 1) {$c_{\overline{00}}$};
		\node [style=none] (21) at (-3, 1.5) {$\ccorner$};
		\node [style=none] (22) at (-1, 1.5) {$\ccorner$};
		\node [style=none] (23) at (1, 1.5) {$\ccorner$};
		\node [style=none] (24) at (3, 1.5) {$\ccorner$};
		\node [style=none] (25) at (2, 2.5) {$\ccorner$};
		\node [style=none] (26) at (-2, 2.5) {$\ccorner$};
		\node [style=none] (27) at (1, 3.5) {$\ccorner$};
		\node [style=none] (28) at (-1, 3.5) {$\ccorner$};
		\node [style=none] (29) at (0, 2.5) {$\ccorner$};
		\node [style=none] (30) at (0, 4.5) {$\ccorner$};
	\end{pgfonlayer}
	\begin{pgfonlayer}{edgelayer}
		\draw [style=arrow] (10) to (12);
		\draw [style=arrow] (10) to (11);
		\draw [style=arrow] (8) to (10);
		\draw [style=arrow] (8) to (13);
		\draw [style=arrow] (13) to (12);
		\draw [style=arrow] (13) to (14);
		\draw [style=arrow] (4) to (8);
		\draw [style=arrow] (1) to (4);
		\draw [style=arrow] (0) to (1);
		\draw [style=arrow] (0) to (2);
		\draw [style=arrow] (1) to (3);
		\draw [style=arrow] (2) to (3);
		\draw [style=arrow] (3) to (7);
		\draw [style=arrow] (7) to (13);
		\draw [style=arrow] (7) to (15);
		\draw [style=arrow] (15) to (14);
		\draw [style=arrow] (6) to (15);
		\draw [style=arrow] (3) to (6);
		\draw [style=arrow] (2) to (5);
		\draw [style=arrow] (5) to (6);
		\draw [style=arrow] (5) to (9);
		\draw [style=arrow] (9) to (17);
		\draw [style=arrow] (6) to (17);
		\draw [style=arrow] (15) to (16);
		\draw [style=arrow] (17) to (16);
		\draw [style=arrow] (17) to (18);
		\draw [style=arrow] (9) to (19);
		\draw [style=arrow] (19) to (18);
		\draw [style=arrow] (19) to (20);
		\draw [style=arrow] (4) to (7);
	\end{pgfonlayer}
\draw[dotted,-] (1) -- (6);
\draw[dotted,-] (2) -- (7);
\draw[dotted,-] (7) -- (16);
\draw[dotted,-] (6) -- (14);
\draw[dotted,-] (1) -- (8);
\draw[dotted,-] (2) -- (9);
\end{tikzpicture}
\end{center}
where each square is a pullback in $\mc{C}$. Such a diagram is to be understood as a composable sequence of spans $$c_{00}\overset{c_{01}}{\nrightarrow}c_{11}\overset{c_{12}}{\nrightarrow}c_{22}\nrightarrow\cdots\nrightarrow c_{nn}$$
where for $i<j<k$ each $c_{ik}$ is the composite (fibre product) of $c_{ii}\overset{c_{ij}}{\nrightarrow}c_{jj}\overset{c_{jk}}{\nrightarrow}c_{kk}$.

Recall that $\Cat_\infty$ is a Cartesian closed $\infty$-category, in particular, there is an internal mapping object bi-functor (cf. \cite[Remark~4.2.1.31]{Lurie:0uBkkKfz}),
\begin{equation}\label{eq:2Yon}\Cat_\infty^{op}\times\Cat_\infty\xrightarrow{(\mc{D},\mc{C})\mapsto \on{Fun}(\mc{D},\mc{C})}\Cat_\infty,\end{equation}
which is separately continuous in either variable. Composing with $\BSigma^{\bullet,\dots,\bullet}:\BDelta^k\to\Cat_\infty$ yields a functor  
$(\BDelta^k)^{op}\times\Cat_\infty\to\Cat_\infty,$ or equivalently, a functor
\begin{subequations}
\begin{equation}\label{eq:barSPAN+}
\begin{split}
\overline{\on{SPAN}}_k^+:\Cat_\infty&\to\on{Fun}\big((\BDelta^k)^{op},\Cat_\infty\big),\\
\mc{C}&\to \bigg[(n_1,\dots,n_k)\to \on{Fun}\big(\BSigma^{n_1,\dots,n_k},\mc{C}\big)\bigg]
\end{split}
\end{equation}
which is continuous (by \cite[Corollary~5.1.2.3]{Lurie:2009un} and the continuity of \eqref{eq:2Yon} in the second variable).


Let $\Cat_\infty^{\on{lex}}\subset \Cat_\infty$ consist of those $\infty$-categories with finite limits and functors preserving finite limits. Suppose that $\mc{C}\in\Cat_\infty^{\on{lex}}$ has finite limits, and $f:\BSigma^{n_1\dots,n_k}\to \mc{C}$ is Cartesian (in the sense of Definition~\ref{def:CartFunct}); then for  any finite limit preserving functor $F:\mc{C}\to\mc{D}$, the composite $F\circ f:\BSigma^{n_1\dots,n_k}\to \mc{D}$ is also Cartesian.
 Therefore, following \cite{Haugseng:2014vw}, we may define
\begin{equation}\label{eq:SPAN+}
\begin{split}
\on{SPAN}_k^+:\Cat_\infty^{\on{lex}}&\to\on{Fun}\big((\BDelta^k)^{op},\Cat_\infty\big),\\
\mc{C}&\to \bigg[(n_1,\dots,n_k)\to \on{Fun}^{\BSigma\textrm{-Cart}}\big(\BSigma^{n_1,\dots,n_k},\mc{C}\big)\bigg]
\end{split}
\end{equation}
\end{subequations}
to be the subfunctor 
 of $\overline{\on{SPAN}}_k^+\rvert_{\Cat_\infty^{\on{lex}}}$ which assigns to each $\mc{C}\in\Cat_\infty^{\on{lex}}$ and each $(n_1,\dots,n_k)\in\BDelta^k$ the full subcategory
spanned by the Cartesian functors $\BSigma^{n_1\dots,n_k}\to \mc{C}$. As explained in \cite{Haugseng:2014vw} the functor \eqref{eq:SPAN+} takes values in $k$-uple category objects (see also \cite{Barwick:2013th}).

\begin{Lemma}\label{lem:SPAN+Cont}
	The functor
	$$\on{SPAN}_k^+:\Cat_\infty^{\on{lex}}\to\on{Cat}^k(\Cat_\infty)$$
	is continuous.
\end{Lemma}
\begin{proof}
	Following Lemma~\ref{lem:ContOfCatObj}, we need only show that the composite 
	\begin{subequations}
	\begin{equation}\label{eq:UnderlyMorph}\Cat_\infty^{\on{lex}}\xrightarrow{\on{SPAN}_k^+}\on{Cat}^k(\Cat_\infty) \xrightarrow{i^*} \on{Fun}\big((\mbb{Morph}^k)^{op},\Cat_\infty\big)\end{equation}
	is continuous, where $i:\mbb{Morph}^k\to \BDelta^k$ is as in Lemma~\ref{lem:ContOfCatObj}. But \eqref{eq:UnderlyMorph} is equivalent to the composite
	\begin{equation}\label{eq:UnderlyMorph2}\Cat_\infty^{\on{lex}}\hookrightarrow\Cat_\infty\xrightarrow{\overline{\on{SPAN}}_k^+}\on{Fun}\big((\BDelta^k)^{op},\Cat_\infty\big) \xrightarrow{i^*} \on{Fun}\big((\mbb{Morph}^k)^{op},\Cat_\infty\big).\end{equation}
	\end{subequations}
	The first arrow in \eqref{eq:UnderlyMorph2} is continuous by \cite{Riehl:2014ut} (or Theorem~\ref{thm:LimInLim}), the second arrow is continuous since $\Cat_\infty$ is Cartesian closed (as explained above), and the final arrow is continuous by \cite[Corollary~5.1.2.3]{Lurie:2009un}.
\end{proof}

Next, let $\iota: \Cat_\infty\to\Spc$ denote the right adjoint to the inclusion, which sends an $\infty$-category $\mc{C}$ to its classifying space of objects, the largest Kan complex contained in $\mc{C}$. Then as in \cite{Haugseng:2014vw}, we define
\begin{equation}\label{eq:SPAN}
\begin{split}
\on{SPAN}_k:\Cat_\infty^{\on{lex}}&\to\on{Cat}^k(\Spc),\\
\mc{C}&\to \bigg[(n_1,\dots,n_k)\to \on{Map}^{\BSigma\textrm{-Cart}}\big(\BSigma^{n_1,\dots,n_k},\mc{C}\big)\bigg]
\end{split}
\end{equation}
to be the composite $\iota\circ\on{SPAN}_k^+$.

Finally let $U_{\on{Seg}}:\on{Cat}^k(\Spc)\to\on{Seg}_k(\Spc)$ denote a right adjoint to $\on{Seg}_k(\Spc)\hookrightarrow\on{Cat}^k(\Spc)$. Then $\on{Span}_k:=U_{\on{Seg}}\circ \on{SPAN}_k$ takes values in complete Segal spaces \cite[Corollary~3.18]{Haugseng:2014vw}.

\begin{Theorem}
	The functor
	$$\on{Span}_k:\Cat_\infty^{\on{lex}}\to\CSS_k(\Spc)$$
	is continuous.	
\end{Theorem}
\begin{proof}
	$\on{Span}_k$ is the composite $U_{\on{Seg}}\circ \iota\circ\on{SPAN}_k^+$, the first two functors are continuous (since they are right adjoints), and the last functor is continuous by Lemma~\ref{lem:SPAN+Cont}.
\end{proof}

\begin{remark}
Let $\mc{K}$ be the subcategory inclusions $i_{n_1,\dots,n_k}:\BLambda^{n_1,\dots,n_k}\to \BSigma^{n_1,\dots,n_k}$ used in the definition of a Cartesian functor (cf. Definition~\ref{def:CartFunct}). Let $\Cat_\infty^\mc{K}\subset \Cat_\infty$ denote the subcategory  consisting of $\infty$-categories which admit all right Kan extensions along any $i_{n_1,\dots,n_k}:\BLambda^{n_1,\dots,n_k}\to \BSigma^{n_1,\dots,n_k}$  and of functors which preserve those right Kan extensions. Then $\Cat_\infty^\mc{K}$ is the
 maximal subcategory of $\Cat_\infty$ on which the functor $\on{SPAN}_k^+$ may be defined. As a consequence of Corollary~\ref{cor:KanClosed}, each of the functors
	\begin{align*}\on{SPAN}_k^+:\Cat_\infty^{\mc{K}}&\to\on{Cat}^k(\Cat_\infty)\\
	\on{SPAN}_k:\Cat_\infty^{\mc{K}}&\to\on{Cat}^k(\Spc)\\
	\on{Span}_k:\Cat_\infty^{\mc{K}}&\to\CSS_k(\Spc)
	\end{align*}
	are continuous.
\end{remark}

\subsection{The sheaf of iterated spans with local systems.}

Suppose that $\mc{X}$ is an $\infty$-topos, and $X_{\bullet,\dots,\bullet}\in \CSS_k(\mc{X})$ is a complete $k$-fold Segal object in $\mc{X}$. In \cite{Haugseng:2014vw}, Haugseng gave an elegant construction of the $(\infty,k)$-category 
$$\on{Span}_k(\mc{X},X_{\bullet,\dots,\bullet})$$ of iterated $k$-fold spans in $\mc{X}$ with local systems valued in $X_{\bullet,\dots,\bullet}$, 
\begin{itemize}
\item whose objects are objects in $\mc{X}$ equipped with a map to the objects of the local system, $X_{\bullet,\dots,\bullet}$,
\item whose morphisms are spans in $\mc{X}$ equipped with compatible maps to the space of morphisms of the local system, $X_{\bullet,\dots,\bullet}$, 
\item $\dots$,
\item and whose $i$-morphisms are $i$-fold spans in $\mc{X}$ equipped with compatible maps to the space of $i$-morphisms of the local system, $X_{\bullet,\dots,\bullet}$.
\end{itemize}



In this section, we show that for any continuous functor $\sigma:\mc{X}^{op}\to \int\CSS_k$ over $$\mc{X}^{op}\xrightarrow{U\mapsto \mc{X}_{/U}}\LTop,$$ 
 the functor
$$\mc{X}^{op}\xrightarrow{U\mapsto \on{Span}_k(\mc{X}_{/U},\sigma(U))}\CSS_k(\hat\Spc)$$
forms an $(\infty,k)$-stack
over $\mc{X}$.

We begin by describing the functor
$$\int\CSS_k\xrightarrow{(\mc{X}, X_{\bullet,\dots,\bullet})\mapsto \on{Span}_k(\mc{X},X_{\bullet,\dots,\bullet})}\CSS_k(\hat\Spc)$$ 
in more detail.

As in \cite{Haugseng:2014vw}, we let $\hat\BSigma\xrightarrow{q} \BDelta^{op}$ denote the Grothendieck fibration classified by the functor $\BSigma^\bullet:\BDelta\xrightarrow{[n]\mapsto \BSigma^n} \Cat$, whose objects are pairs $\big([n],(i,j)\big)$ with $0\leq i\leq j\leq n$, and whose morphisms 
\begin{equation}\label{eq:hatBSig}\big([n],(i,j)\big)\xrightarrow{\begin{array}{rcl}[n]&\leftarrow &[m]:\phi\\(i,j)&\rightarrow &\big(\phi(i'),\phi(j')\big)\end{array}}\big([m],(i',j')\big)\end{equation} are pairs of morphisms $\phi:[m]\to[n]$ in $\BDelta$ and $(i,j)\to \big(\phi(i'),\phi(j')\big)$ in $\BSigma^n$.

Let $\imath^*\mc{Z}\to \LTop$ denote the (cannonical) presentable fibration classified by the inclusion $\imath:\LTop\hookrightarrow\widehat{\Cat_\infty}$. The functor $\overline{\on{SPAN}}^+_k:\LTop\times(\BDelta^k)^{op}\to \widehat{\Cat_\infty}$ classifies the coCartesian fibration (cf. \cite[Corollary~3.2.2.13]{Lurie:2009un}) $$(\imath^*\mc{Z}\times(\BDelta^k)^{op})^{\LTop\times \hat\BSigma^k}\to \LTop\times(\BDelta^k)^{op},$$
whose fibre over any $\big(\mc{X};(n_1,\dots,n_k)\big)\in \LTop\times(\BDelta^k)^{op}$ is equivalent to
$$\on{Fun}\big(\BSigma^{n_1,\dots,n_k},\mc{X}\big).$$
Similarly, $\on{SPAN}^+_k:\LTop\times(\BDelta^k)^{op}\to \widehat{\Cat_\infty}$ classifies the coCartesian fibration defined as the full subcategory $${\int\on{SPAN}^+_k}\subset (\imath^*\mc{Z}\times(\BDelta^k)^{op})^{\LTop\times \hat\BSigma^k}$$
spanned (over $\big(\mc{X};(n_1,\dots,n_k)\big)\in \LTop\times(\BDelta^k)^{op}$) by the Cartesian functors $\BSigma^{n_1,\dots,n_k}\to \mc{X}$ (cf. Definition~\ref{def:CartFunct}).

There is a second functor $$\Pi:\hat\BSigma\xrightarrow{\big([n],(i,j)\big)\mapsto [j-i]} \BDelta^{op}$$ which sends the map \eqref{eq:hatBSig}
to $$[j-i]\xrightarrow{k\mapsto\phi(k+i')-i}[j'-i'].$$

The corresponding morphism of Cartesian fibrations 
$$\begin{tikzpicture}
	\mmat{m}{\BDelta^{op}\times\BDelta^{op}&&\hat\BSigma\\&\BDelta^{op}&\\};
	\draw[<-] (m-1-1) -- node[swap]{ $\Pi\times q$} (m-1-3);
	\draw[->] (m-1-1) -- (m-2-2);
	\draw[->] (m-1-3) -- (m-2-2);
\end{tikzpicture}$$
induces a morphism of coCartesian fibrations
$$\begin{tikzpicture}
	\mmat{m}{(\imath^*\mc{Z}\times(\BDelta^k)^{op})^{\LTop\times (\BDelta^k)^{op}\times (\BDelta^k)^{op}}&&(\imath^*\mc{Z}\times(\BDelta^k)^{op})^{\LTop\times \hat\BSigma^k}\\
	&\LTop\times(\BDelta^k)^{op}&\\};
	\draw[->] (m-1-1) -- (m-1-3);
	\draw[->] (m-1-1) -- (m-2-2);
	\draw[->] (m-1-3) -- (m-2-2);
\end{tikzpicture}$$

By \cite[Lemma~4.3]{Haugseng:2014vw} this restricts to a morphism:
\begin{equation}\label{eq:sectS}\begin{tikzpicture}
	\mmat{m}{\bigg({\int\CSS_k}\bigg)\times(\BDelta^k)^{op}&&{\int\on{SPAN}^+_k}\\
	&\LTop\times(\BDelta^k)^{op}&\\};
	\draw[->] (m-1-1) -- node {$s_0$} (m-1-3);
	\draw[->] (m-1-1) -- (m-2-2);
	\draw[->] (m-1-3) -- (m-2-2);
\end{tikzpicture}\end{equation}
which sends any $\big(\mc{X};(n_1,\dots,n_k)\big)\in \LTop\times(\BDelta^k)^{op}$ to 
$$X_{\bullet,\dots,\bullet}\circ \Pi_{n_1,\dots,n_k}\in \on{Fun}^{\BSigma\textrm{-Cart}}(\BSigma^{n_1,\dots,n_k},\mc{X}),$$
where $\Pi_{n_1,\dots,n_k}:=\Pi\rvert_{\BSigma^{n_1,\dots,n_k}}$.

In turn, \eqref{eq:sectS} defines a section of the left hand arrow in the pullback square 
$$\begin{tikzpicture}
	\mmat{m}{\mc{Q} & {\int\on{SPAN}^+_k}\\
	\bigg({\int\CSS_k}\bigg)\times(\BDelta^k)^{op} & \LTop\times(\BDelta^k)^{op}\\};
	\draw[->] (m-1-1) -- (m-1-2);
	\draw[->] (m-1-1) -- (m-2-1);
	\draw[->] (m-1-2) -- (m-2-2);
	\draw[->] (m-2-1) -- (m-2-2);
	\draw[dotted,->] (m-2-1) to[bend left] node{$s$}(m-1-1);
\end{tikzpicture}$$
For brevity, we denote $\mc{D}=\big({\int\CSS_k}\big)\times(\BDelta^k)^{op}$, and we define $\mc{Q}^{/s}\to \mc{D}\cong\big({\int\CSS_k}\big)\times(\BDelta^k)^{op}$ to be the simplicial set satisfying the universal property that
for any morphism of simplicial sets $Y\to \mc{D}$, 
 commutative diagrams of the form 
$$\begin{tikzpicture}
	\mmat{m}{\mc{D} &\mc{Q}\\
	Y\diamond_{\mc{D}}\mc{D}& \mc{D}\\};
	\draw[->] (m-1-1) -- node {$s$} (m-1-2);
	\draw[right hook->] (m-1-1) --  (m-2-1);
	\draw[->] (m-1-2) --  (m-2-2);
	\draw[->] (m-2-1) -- (m-2-2);
	\draw[->] (m-2-1) --  (m-1-2);
\end{tikzpicture}$$
(where $Y\diamond_{\mc{D}}\mc{D}=Y\times\Delta^1\coprod_{Y\times\{1\}}\mc{D}$)
correspond to diagrams of the form
$$\begin{tikzpicture}\node (y) at (-2,1) {$Y$};
\node (q) at (2,1) {$\mc{Q}^{/s}$};
\node (css) at (0,0) {$\mc{D}$};
\draw[->] (y) -- (q);
\draw[->] (y) -- (css);
\draw[->] (q) -- (css);
\end{tikzpicture}$$

By \cite[Proposition~4.2.2.4.]{Lurie:2009un}  $\mc{Q}^{/s}\to \big({\int\CSS_k}\big)\times(\BDelta^k)^{op}$ is a coCartesian fibration whose fibre over $\big(\mc{X},X_{\bullet,\dots,\bullet};(n_1,\dots,n_k)\big)\in {\int\CSS_k}\times(\BDelta^k)^{op}$ is equivalent to the overcategory
$$\on{Fun}^{\BSigma\textrm{-Cart}}(\BSigma^{n_1,\dots,n_k},\mc{X})_{/X_{\bullet,\dots,\bullet}\circ \Pi_{n_1,\dots,n_k}}$$

Let $$(\on{SPAN}^+_k)^{/s}:{\int\CSS_k}\to \on{Fun}\big((\BDelta^k)^{op},\widehat{\Cat_\infty}\big)$$ be a functor classifying $\mc{Q}^{/s}$. By, \cite[Proposition~4.5]{Haugseng:2014vw}, $(\on{SPAN}^+_k)^{/s}$ takes values in $k$-uple category objects in $\widehat{\Cat_\infty}$. We define $$\on{Span}_k:=U_{\on{Seg}}\circ \iota\circ (\on{SPAN}^+_k)^{/s}:{\int\CSS_k}\to \on{Seg}_k(\widehat{\Spc}).$$
It follows from \cite[Proposition~4.8]{Haugseng:2014vw} that $\on{Span}_k$ takes values in complete $k$-fold Segal spaces, i.e. we have a functor
\begin{equation}\label{eq:locSpan}\on{Span}_k:{\int\CSS_k}\to \CSS_k(\widehat{\Spc})\end{equation}
which sends a complete $k$-fold Segal object $X_{\bullet,\dots,\bullet}\in\CSS_k(\mc{X})$ to the complete $k$-fold Segal space $$\on{Span}_k(\mc{X},X_{\bullet,\dots,\bullet})$$ of iterated spans in $\mc{X}$ with local systems valued in $X_{\bullet,\dots,\bullet}$ (cf. \cite[\S~4]{Haugseng:2014vw}).

\subsubsection{Continuity of $(\mc{X},X_{\bullet,\dots,\bullet})\to \on{Span}_k(\mc{X},X_{\bullet,\dots,\bullet})$.}

\begin{Lemma}\label{lem:SpanProduct}
	The functor \eqref{eq:locSpan} preserves small products.
\end{Lemma}
\begin{proof}
Since $\int^{\acute{e}t}\CSS_k\subseteq\int\CSS_k$ is a continuous inclusion of a wide subcategory (i.e. it contains all the objects), it suffices to show that the restriction of \eqref{eq:locSpan} to $\int^{\acute{e}t}\CSS_k$ preserves small products.

	Now suppose $\{X^j_{\bullet,\dots,\bullet}\in \CSS_k(\mc{X}^j)\}_{j\in J}$ is a set of complete Segal objects indexed by a small set $J$. Since $p:\int^{\acute{e}t}\CSS_k\to \LTop_{\acute{e}t}$ is a presentable fibration and $\LTop_{\acute{e}t}$ has small products, we may compute the product $$\prod_{j\in J}(\mc{X}^j,X^j_{\bullet,\dots,\bullet})\in \int\CSS_k$$ by first computing the product $\prod_{j\in J}\mc{X}^j$ in $\LTop_{\acute{e}t}$, and then computing the $p$-relative product of $\{X^j_{\bullet,\dots,\bullet}\in \CSS_k(\mc{X}^j)\}_{j\in J}$ over $\prod_{j\in J}\mc{X}^j$. 
	
	Note that \cite[Proposition~6.3.2.3 and Theorem~6.3.5.13]{Lurie:2009un} imply that $\prod_{j\in J}\mc{X}^j$ is just the product of the $\infty$-categories $\mc{X}^j$ (i.e. we can take this product in $\widehat{\Cat_\infty}$ rather than $\LTop_{\acute{e}t}$). Next, Theorem~\ref{thm:CSSaSheaf} implies that the fibre of $\int^{\acute{e}t}\CSS_k$ over $\prod_{j\in J} \mc{X}^j$ is just 
	$$\CSS_k(\prod_{j\in J} \mc{X}^j)\cong \prod_{j\in J} \CSS_k(\mc{X}^j),$$
	where the right hand product is taken in $\widehat{\Cat_\infty}$.
	Consequently, the $p$-relative product of $\{X^j_{\bullet,\dots,\bullet}\}_{j\in J}$ over $\prod_{j\in J} \mc{X}^j$ is $$\prod_{j\in J}X^j_{\bullet,\dots,\bullet}\in \prod_{j\in J} \CSS_k(\mc{X}^j).$$


	Next we argue that the restriction of $(\on{SPAN}^+_k)^{/s}$ to $\int^{\acute{e}t}\CSS_k$ preserves small products. In view of Lemma~\ref{lem:ContOfCatObj}, we need only show that for any $0\leq i_1,\dots, i_k\leq 1$, the functor
	\begin{equation}\label{eq:overCatSpan1Prod}\int^{\acute{e}t}\CSS_k\xrightarrow{(\mc{X},X_{\bullet,\dots,\bullet})\mapsto \on{Fun}(\BSigma^{i_1,\dots,i_k},\mc{X})_{/X_{\bullet,\dots,\bullet}\circ \Pi_{i_1,\dots,i_k}}}\widehat{\Cat_\infty}.\end{equation}
	preserves small products.
	
	Notice that the continuous functor 
	$$\prod_{j\in J} \CSS_k(\mc{X}^j)\cong \CSS_k(\prod_{j\in J} \mc{X}^j)\hookrightarrow \on{Fun}\big((\BDelta^k)^{op},\prod_{j\in J} \mc{X}^j)\xrightarrow{\Pi_{i_1,\dots,i_k}^*} \on{Fun}(\BSigma^{i_1,\dots,i_k},\prod_{j\in J} \mc{X}^j)$$
	 takes $\prod_{j\in J}X^j_{\bullet,\dots,\bullet}$ to $$\big(\prod_{j\in J}X^j_{\bullet,\dots,\bullet}\big)\circ \Pi_{i_1,\dots,i_k}\cong \prod_{j\in J}\big(X^j_{\bullet,\dots,\bullet}\circ \Pi_{i_1,\dots,i_k}\big).$$
	So \begin{multline*}\on{Fun}(\BSigma^{i_1,\dots,i_k},\prod_{j\in J}\mc{X}^j)_{/\big((\prod_{j\in J}X^j_{\bullet,\dots,\bullet})\circ \Pi_{i_1,\dots,i_k}\big)}\\\cong \big(\prod_{j\in J}\on{Fun}(\BSigma^{i_1,\dots,i_k},\mc{X}^j)\big)_{/\big(\prod_{j\in J}(X^j_{\bullet,\dots,\bullet}\circ \Pi_{i_1,\dots,i_k})\big)}\\
 	\cong \prod_{j\in J}\bigg(\on{Fun}(\BSigma^{i_1,\dots,i_k},\mc{X}^j)_{/X^j_{\bullet,\dots,\bullet}\circ \Pi_{i_1,\dots,i_k}}\bigg),
 \end{multline*}
which implies that \eqref{eq:overCatSpan1Prod} preserves small products.

Finally, we have $\on{Span}_k= U_{\on{Seg}}\circ\iota\circ (\on{SPAN}^+_k)^{/s}$, and since $U_{\on{Seg}}$ and $\iota$ are both right adjoints, they are continuous, which implies the statement we wished to prove.
	
\end{proof}

\begin{Lemma}\label{lem:SpanPullback}
The functor \eqref{eq:locSpan} preserves pullbacks.
\end{Lemma}
\begin{proof}
We begin by arguing that $(\on{SPAN}^+_k)^{/s}:\int\CSS_k\to \widehat{\Cat_\infty}$ preserves pullbacks. In view of Lemma~\ref{lem:ContOfCatObj}, we need only show that for any $0\leq i_1,\dots, i_k\leq 1$, the functor
	\begin{equation}\label{eq:overCatSpan1Pullback}\int\CSS_k\xrightarrow{(\mc{X},X_{\bullet,\dots,\bullet})\mapsto \on{Fun}(\BSigma^{i_1,\dots,i_k},\mc{X})_{/X_{\bullet,\dots,\bullet}\circ \Pi_{i_1,\dots,i_k}}}\widehat{\Cat_\infty}.\end{equation}
	preserves pullbacks.

	Suppose we have a diagram 
	\begin{subequations}\begin{equation}\label{eq:CSSpullback}(\mc{X},X_{\bullet,\dots,\bullet})\rightarrow (\mc{Z},Z_{\bullet,\dots,\bullet})\leftarrow(\mc{Y},Y_{\bullet,\dots,\bullet})\end{equation} in $\int\CSS_k$ (here $X_{\bullet,\dots,\bullet}\in \CSS_k(\mc{X})$, $Y_{\bullet,\dots,\bullet}\in\CSS_k(\mc{Y})$, and $Z_{\bullet,\dots,\bullet}\in \CSS_k(\mc{Z})$).
	To compute the pullback of \eqref{eq:CSSpullback} we use \cite[Corollary~4.3.1.11]{Lurie:2009un}. That is, we first compute the pullback
	\begin{equation}\label{eq:CSSpullback1}
	\begin{tikzpicture}
	\mmat{m}{\mc{W} &\mc{Y}\\ \mc{X}& \mc{Z}\\};
	\draw[->] (m-1-1) -- node {$b^*$} (m-1-2);
	\draw[->] (m-1-1) -- node[swap] {$a^*$}  (m-2-1);
	\draw[->] (m-1-2) -- node {$g^*$} (m-2-2);
	\draw[->] (m-2-1) -- node[swap] {$f^*$} (m-2-2);
	\draw[->] (m-1-1) -- node {$c^*$} (m-2-2);
	\end{tikzpicture}\end{equation} in $\LTop$ and then take the relative pullback 
	\begin{equation}\label{eq:CSSpullback2}
	\begin{tikzpicture}
	\node (m-1-1) at (-4,2) {$\overset{W}{\overbrace{(a_*)_!X\times_{(c_*)_!Z}(b_*)_!Y}}$};
	\node (m-1-2) at (-1,2) {$Y$};
	\node (m-2-1) at (-4,0) {$X$};
	\node (m-2-2) at (-1,0) {$Z$};
	\node (m-1-1') at (2,2) {$\CSS_k(\mc{W})$};
	\node (m-1-2') at (5,2) {$\CSS_k(\mc{Y})$};
	\node (m-2-1') at (2,0) {$\CSS_k(\mc{X})$};
	\node (m-2-2') at (5,0) {$\CSS_k(\mc{Z})$};
	\node (in) at (.5,1) {$\in $};
	\draw[->] (m-1-1) --  (m-1-2);
	\draw[->] (m-1-1) --  (m-2-1);
	\draw[->] (m-1-2) --  (m-2-2);
	\draw[->] (m-2-1) --  (m-2-2);
	\draw[->] (m-1-1') --  (m-1-2');
	\draw[->] (m-1-1') --  (m-2-1');
	\draw[->] (m-1-2') --  (m-2-2');
	\draw[->] (m-2-1') --  (m-2-2');
	\end{tikzpicture}\end{equation}\end{subequations} in the fibre $\CSS_k(\mc{W})$ over $\mc{W}$ (here we have dropped the abstract multi-indices on $X_{\bullet,\dots,\bullet}$, $Y_{\bullet,\dots,\bullet}$, and $Z_{\bullet,\dots,\bullet}$).
	
	For any $\infty$-category $\mc{C}$, let \begin{subequations}\begin{equation}\label{eq:preEmb}(\Cat_\infty)_{/\mc{C}}\hookleftarrow\Pre(\mc{C})\end{equation} denote the functor which sends a presheaf over $\mc{C}$ to the corresponding right fibration over $\mc{C}$. Then \cite[Corollary 2.1.2.10]{Lurie:2009un} implies that \eqref{eq:preEmb} is equivalent to a reflective left localization of $(\Cat_\infty)_{/\mc{C}}$; in particular \eqref{eq:preEmb} is continuous (see also \cite[Theorem~4.5]{Gepner:2015ww}). Since the Yoneda embedding is continuous, and the forgetful functor $(\Cat_\infty)_{/\mc{C}}\to\Cat_\infty$ preserves pullbacks, it follows that the composite
	\begin{equation}\label{eq:pullbackPreservingYon}
		\begin{tikzpicture}
			\mmat[1em]{m}{\mc{C}&\Pre(\mc{C})&(\Cat_\infty)_{/\mc{C}}&\Cat_\infty\\
			c&&&\mc{C}_{/c}\\};
			\draw[->] (m-1-1) -- node {$\yon$} (m-1-2);
			\draw[right hook->] (m-1-2) -- (m-1-3);
			\draw[->] (m-1-3) --  (m-1-4);
			\draw[->] (m-2-1) -- (m-2-4);
		\end{tikzpicture}
	\end{equation}
	\end{subequations}
	preserves pullbacks.
	
	Applying \eqref{eq:overCatSpan1Pullback} to $W_{\bullet,\dots,\bullet}$ (the top left corner of \eqref{eq:CSSpullback2}) and using the continuity of \eqref{eq:pullbackPreservingYon}  yields a pullback diagram
	\begin{equation}\label{eq:pullbackSpan+}\begin{tikzpicture}
	\mmat{m}{\on{Fun}(\BSigma^{i_1,\dots,i_k},\mc{W})_{/W_{\bullet,\dots,\bullet}\circ \Pi_{i_1,\dots,i_k}} &\on{Fun}(\BSigma^{i_1,\dots,i_k},\mc{W})_{/(b_*)_!Y_{\bullet,\dots,\bullet}\circ \Pi_{i_1,\dots,i_k}}\\ \on{Fun}(\BSigma^{i_1,\dots,i_k},\mc{W})_{/(a_*)_!X_{\bullet,\dots,\bullet}\circ \Pi_{i_1,\dots,i_k}}& \on{Fun}(\BSigma^{i_1,\dots,i_k},\mc{W})_{/(c_*)_!Z_{\bullet,\dots,\bullet}\circ \Pi_{i_1,\dots,i_k}}\\};
	\draw[->] (m-1-1) -- (m-1-2);
	\draw[->] (m-1-1) -- (m-2-1);
	\draw[->] (m-1-2) -- (m-2-2);
	\draw[->] (m-2-1) -- (m-2-2);
	\end{tikzpicture}\end{equation}
	Now, since $$(a^*)_!:\on{Fun}(\BSigma^{i_1,\dots,i_k},\mc{W})\leftrightarrows\on{Fun}(\BSigma^{i_1,\dots,i_k},\mc{X}):(a_*)_!$$ is an adjunction, we have a pullback square
	$$\begin{tikzpicture}
	\mmat{m}{\on{Fun}(\BSigma^{i_1,\dots,i_k},\mc{W})_{/(a_*)_!X_{\bullet,\dots,\bullet}\circ \Pi_{i_1,\dots,i_k}} &\on{Fun}(\BSigma^{i_1,\dots,i_k},\mc{X})_{/X_{\bullet,\dots,\bullet}\circ \Pi_{i_1,\dots,i_k}}\\ \on{Fun}(\BSigma^{i_1,\dots,i_k},\mc{W})& \on{Fun}(\BSigma^{i_1,\dots,i_k},\mc{X})\\};
	\draw[->] (m-1-1) -- (m-1-2);
	\draw[->] (m-1-1) -- (m-2-1);
	\draw[->] (m-1-2) -- (m-2-2);
	\draw[->] (m-2-1) -- node {$(a^*)_!$} (m-2-2);
	\end{tikzpicture}$$
	Similarly, the right hand terms of \eqref{eq:pullbackSpan+} fit into analagous pullback squares. It follows that $$\on{Fun}(\BSigma^{i_1,\dots,i_k},\mc{W})_{/W_{\bullet,\dots,\bullet}\circ \Pi_{i_1,\dots,i_k}}$$ fits into  a limit diagram
	\begin{equation}\label{eq:pullbackSpan+1}\begin{tikzpicture}
	\node (m-1-1-1) at (-2,4) {$\on{Fun}(\BSigma^{I},\mc{W})_{/W\circ \Pi_I}$};
	\node (m-1-1-2) at (-4,2) {$\on{Fun}(\BSigma^{I},\mc{X})_{/X\circ \Pi_I}$};
	\node (m-1-2-1) at (4,4) {$\on{Fun}(\BSigma^{I},\mc{Y})_{/Y\circ \Pi_I}$};
	\node (m-1-2-2) at (2,2) {$\on{Fun}(\BSigma^{I},\mc{Z})_{/Z\circ \Pi_I}$};
	\node (m-2-1-1) at (-2,0) {$\on{Fun}(\BSigma^{I},\mc{W})$};
	\node (m-2-1-2) at (-4,-2) {$\on{Fun}(\BSigma^{I},\mc{X})$};
	\node (m-2-2-1) at (4,0) {$\on{Fun}(\BSigma^{I},\mc{Y})$};
	\node (m-2-2-2) at (2,-2) {$\on{Fun}(\BSigma^{I},\mc{Z})$};
	\draw[->] (m-1-1-1) -- (m-1-1-2);
	\draw[->] (m-1-1-1) -- (m-1-2-1);
	\draw[->] (m-1-1-1) -- (m-2-1-1);
	\draw[->] (m-1-1-2) -- (m-1-2-2);
	\draw[->] (m-1-1-2) -- (m-2-1-2);
	\draw[->] (m-1-2-1) -- (m-1-2-2);
	\draw[->] (m-1-2-1) -- (m-2-2-1);
	\draw[->] (m-2-1-1) -- node[swap] {$a^*_!$} (m-2-1-2);
	\draw[->] (m-2-1-1) -- node {$b^*_!$} (m-2-2-1);
	\draw[->] (m-1-2-2) -- (m-2-2-2);
	\draw[->] (m-2-1-2) -- node[swap] {$f^*_!$}(m-2-2-2);
	\draw[->] (m-2-2-1) -- node {$g^*_!$} (m-2-2-2);
\end{tikzpicture}\end{equation}
where we have abbreviated the multi-index $i_1,\dots,i_k=I$ and dropped the abstract multi-indices on $W_{\bullet,\dots,\bullet}$, $X_{\bullet,\dots,\bullet}$, $Y_{\bullet,\dots,\bullet}$, and $Z_{\bullet,\dots,\bullet}$.

Since the bottom square in \eqref{eq:pullbackSpan+1} is already a pullback square, it follows that the top square is also a pullback square, which proves that \eqref{eq:overCatSpan1Pullback} preserves pullbacks.

Finally, we have $\on{Span}_k= U_{\on{Seg}}\circ\iota\circ (\on{SPAN}^+_k)^{/s}$, and since $U_{\on{Seg}}$ and $\iota$ are both right adjoints, they are continuous, which implies the statement we wished to prove.
	
\end{proof}

\begin{Theorem}\label{thm:SpanCont}
The functor \eqref{eq:locSpan},
$$\int\CSS_k\xrightarrow{(\mc{X}, X_{\bullet,\dots,\bullet})\mapsto \on{Span}_k(\mc{X},X_{\bullet,\dots,\bullet})}\CSS_k(\hat\Spc),$$ 
preserves small limits.	
\end{Theorem}
\begin{remark}[Warning!]
	The inclusion of a fibre $\CSS_k(\mc{X})\hookrightarrow\int\CSS_k$ doesn't preserve products or terminal objects (though it does preserve small limits with connected diagrams). So the functor  $$\CSS_k(\mc{X})\xrightarrow{X_{\bullet,\dots,\bullet}\mapsto \on{Span}_k(\mc{X},X_{\bullet,\dots,\bullet})}\CSS_k(\hat\Spc)$$ is \emph{not} continuous: while it does preserve small limits with connected diagrams, it generally fails to preserve products or terminal objects.
\end{remark}
\begin{proof}
According to \cite[Proposition~4.4.2.7]{Lurie:2009un}, it suffices to prove this result for pullbacks and small products; thus the result follows from Lemmas~\ref{lem:SpanProduct} and \ref{lem:SpanPullback}.
\end{proof}

\begin{Theorem}\label{thm:StackofSpans}
	Suppose that  $\mc{X}$ is an $\infty$-topos  and $\sigma:\mc{X}^{op}\to \int\CSS_k$ is a continuous functor fitting into the diagram 
	$$\begin{tikzpicture}
	\node (x) at (-2,1) {$\mc{X}^{op}$};
	\node (css) at (2,1) {$\int\CSS_k$};
	\node (ltop) at (0,-1) {$\LTop$};
	\draw[->] (x) -- node {$\sigma$} (css);
	\draw[->] (x) -- (ltop);
	\draw[->] (css) -- (ltop);
\end{tikzpicture}$$
where the left diagonal arrow $\mc{X}^{op}\xrightarrow{U\mapsto \mc{X}_{/U}}\LTop$ is \eqref{eq:FXtopostoEtale}. Then 
\begin{equation}\label{eq:SpanLocStack}\on{Span}_k\circ\sigma:\mc{X}^{op}\xrightarrow{U\mapsto \on{Span}_k(\mc{X}_{/U},\sigma(U))}\CSS_k(\hat\Spc)\end{equation}
forms an $(\infty,k)$-stack
over $\mc{X}$.

In particular, given
any complete $k$-fold Segal space $X_{\bullet,\dots,\bullet}\in\CSS_k(\mc{X})$, iterated spans in $\mc{X}$ with local systems valued in $X_{\bullet,\dots,\bullet}$ form an $(\infty,k)$-stack
\begin{equation}\label{eq:ItSpanLocSheaf}\mc{X}^{op}\xrightarrow{U\mapsto \on{Span}_k(\mc{X}_{/U},U\times X_{\bullet,\dots,\bullet})}\CSS_k(\hat\Spc)\end{equation}
over $\mc{X}$.
\end{Theorem}
\begin{proof}
The first statement is equivalent to the continuity of \eqref{eq:SpanLocStack}, which follows directly from Theorem~\ref{thm:SpanCont}.

Let $F:\mc{X}^{op}\to\LTop_{\acute{e}t}$ be defined by \eqref{eq:FXtopostoEtale}. Of course, we have $\mc{X}^{op}\cong (\LTop_{\acute{e}t})_{\mc{X}/}$, so $F^*\big(\int^{\acute{e}t}\CSS_k\big)$ is equivalent to the pullback
$$\begin{tikzpicture}
	\mmat{m}{F^*\big(\int^{\acute{e}t}\CSS_k\big)&\int^{\acute{e}t}\CSS_k\\
	(\LTop_{\acute{e}t})_{\mc{X}/}&\LTop_{\acute{e}t}\\};
	\draw[->] (m-1-1) --  (m-1-2);
	\draw[->] (m-1-1) --  (m-2-1);
	\draw[->] (m-1-2) --  (m-2-2);
	\draw[->] (m-2-1) -- node[swap] {$F$} (m-2-2);
\end{tikzpicture}$$  
Since $\mc{X}\in (\LTop_{\acute{e}t})_{\mc{X}/}$ is an initial object, \cite[Proposition~3.3.3.1]{Lurie:2009un} and Theorem~\ref{thm:CSSaSheaf} imply that the $\infty$-category of coCartesian sections of $F^*\big(\int^{\acute{e}t}\CSS_k\big)\to (\LTop_{\acute{e}t})_{\mc{X}/}\cong \mc{X}^{op}$ is equivalent to $\CSS_k(\mc{X})$. In particular, any complete $k$-fold complete Segal object $X_{\bullet,\dots,\bullet}\in\CSS_k(\mc{X})$ determines a coCartesian section $\sigma:\mc{X}^{op}\to F^*\big(\int^{\acute{e}t}\CSS_k\big)$. By \cite[Lemma~6.3.3.5]{Lurie:2009un}, $\sigma$ is continuous, which implies that \eqref{eq:ItSpanLocSheaf} is an $(\infty,k)$-stack.
\end{proof}

\bibliography{basicbib}{}
\bibliographystyle{plain}
\end{document}